\RequirePackage{amsmath}
\documentclass[smallextended,final]{svjour3}

\makeatletter
\let\cl@chapter\undefined
\makeatother

\usepackage{lipsum}
\usepackage{amsfonts}
\usepackage{graphicx}
\usepackage{epstopdf}
\usepackage{algorithmic}
\usepackage[capitalize]{cleveref}
\usepackage{amsfonts}%
\usepackage{amssymb}%
\usepackage{mathrsfs}
\usepackage{overpic}
\usepackage{enumitem}
\usepackage{mathabx}
\usepackage{xcolor}
\usepackage{url}
\usepackage{geometry}

\usepackage{amsopn}

\newcommand{\R}{\mathbb{R}}
\newcommand{\E}{\mathbb{E}}

\renewcommand{\L}{\mathcal{L}}
\renewcommand{\d}{\textup{d}}
\newcommand{\F}{\mathscr{F}}
\newcommand{\mtx}[1]{\mathbf{#1}}
\DeclareMathOperator{\Tr}{Tr}

\DeclareMathOperator{\diam}{diam}
\DeclareMathOperator{\dist}{dist}
\DeclareMathOperator{\cov}{cov}

\DeclareMathOperator{\HS}{HS}
\renewcommand{\d}{\,\textup{d}}

\newcommand{\at}[1]{{\color{red}{AT: #1}}} 

\smartqed  
\usepackage{etoolbox}
\AtEndEnvironment{proof}{\phantom{}\qed} 

\journalname{Foundations of Computational Mathematics}

\begin{document}

\title{Learning elliptic partial differential equations with randomized linear algebra
\thanks{Communicated by Arieh Iserles.\\
\newline
This work is supported by the EPSRC Centre for Doctoral Training in Industrially Focused Mathematical Modelling (EP/L015803/1) in collaboration with Simula Research Laboratory and by the National Science Foundation grants DMS-1818757, DMS-1952757, and DMS-2045646.}
}

\titlerunning{Learning elliptic PDEs with randomized linear algebra}

\author{Nicolas Boull\'e  \and Alex Townsend}


\institute{N. Boull\'e \at
              Mathematical Institute, University of Oxford, Oxford, OX2 6GG, UK \\
              \email{boulle@maths.ox.ac.uk}
           \and
           A. Townsend \at
             Department of Mathematics, Cornell University, Ithaca, NY 14853, USA \\
             \email{townsend@cornell.edu}
}

\date{Received: 1 February 2021 / Revised: 18 November 2021 / Accepted: 20 November 2021}

\maketitle

\begin{abstract}
Given input-output pairs of an elliptic partial differential equation (PDE) in three dimensions, we derive the first theoretically-rigorous scheme for learning the associated Green's function $G$. By exploiting the hierarchical low-rank structure of $G$, we show that one can construct an approximant to $G$ that converges almost surely and achieves a relative error of $\mathcal{O}(\Gamma_\epsilon^{-1/2}\log^3(1/\epsilon)\epsilon)$ using at most $\mathcal{O}(\epsilon^{-6}\log^4(1/\epsilon))$ input-output training pairs with high probability, for any $0<\epsilon<1$. The quantity $0<\Gamma_\epsilon\leq 1$ characterizes the quality of the training dataset. Along the way, we extend the randomized singular value decomposition algorithm for learning matrices to Hilbert--Schmidt operators and characterize the quality of covariance kernels for PDE learning.
\end{abstract}

\section{Introduction}
Can one learn a differential operator from pairs of solutions and righthand sides? If so, how many pairs are required? These two questions have received significant research attention~\cite{feliu2020meta,li2020fourier,long2018pde,pang2019neural}. From data, one hopes to eventually learn physical laws of nature or conservation laws that elude scientists in the biological sciences~\cite{yazdani2020systems}, computational fluid dynamics~\cite{raissi2020hidden}, and computational physics~\cite{raissi2018deep}. The literature contains many highly successful practical schemes based on deep learning techniques~\cite{meng2020ppinn,raissi2019physics}. However, the challenge remains to understand when and why deep learning is effective theoretically. This paper describes the first theoretically-justified scheme for discovering scalar-valued elliptic partial differential equations (PDEs) in three variables from input-output data and provides a rigorous learning rate. While our novelties are mainly theoretical, we hope to motivate future practical choices in PDE learning.

We suppose that there is an unknown second-order uniformly elliptic linear PDE operator\footnote{Here, $L^2(D)$ is the space of square-integrable functions defined on $D$, $\mathcal{H}^k(D)$ is the space of $k$ times weakly differentiable functions in the $L^2$-sense, and $\mathcal{H}^1_0(D)$ is the closure of $\mathcal{C}_c^\infty(D)$ in $\mathcal{H}^1(D)$. Here, $\mathcal{C}_c^\infty(D)$ is the space of infinitely differentiable compactly supported functions on $D$. Roughly speaking, $\mathcal{H}^1_0(D)$ are the functions in $\mathcal{H}^1(D)$ that are zero on the boundary of $D$.} $\mathcal{L}:\mathcal{H}^2(D)\cap\mathcal{H}_0^1(D) \to L^2(D)$ with a bounded domain $D\subset\mathbb{R}^3$ with Lipschitz smooth boundary~\cite{evans10}, which takes the form
\begin{equation} 
\mathcal{L}u(x) = -\nabla \cdot \left(A(x) \nabla u\right)+c(x)\cdot\nabla u+d(x)u, \quad x\in D, \quad u|_{\partial D} = 0.
\label{eq:PDEForm}
\end{equation}
Here, for every $x\in D$, we have that $A(x)\in\mathbb{R}^{3\times 3}$ is a symmetric positive definite matrix with bounded coefficient functions so that\footnote{For $1\leq r\leq \infty$, we denote by $L^r(D)$ the space of functions defined on the domain $D$ with finite $L^r$ norm, where $\|f\|_{r}$ = $(\int_D|f|^r\d x)^{1/r}$ if $r<\infty$, and $\|f\|_{\infty}=\inf\{C>0:|f(x)|\leq C \text{ for almost every }x\in D\}$.} $A_{ij}\in L^{\infty}(D)$, $c\in L^r(D)$ with $r\geq 3$, $d\in L^s(D)$ for $s\geq 3/2$, and $d(x)\geq 0$~\cite{kim2019green}. We emphasize that the regularity requirements on the variable coefficients are quite weak. 

The goal of PDE learning is to discover the operator $\mathcal{L}$ from $N\geq 1$ input-output pairs, i.e., $\{(f_j,u_j)\}_{j=1}^N$, where $\mathcal{L}u_j = f_j$ and $u_j|_{\partial D} = 0$ for $1\leq j\leq N$. There are two main types of PDE learning tasks: (1) Experimentally-determined input-output pairs, where one must do the best one can with the predetermined information and (2) Algorithmically-determined input-output pairs, where the data-driven learning algorithm can select $f_1,\ldots,f_N$ for itself. In this paper, we focus on the PDE learning task where we have algorithmically-determined input-output pairs.  In particular, we suppose that the functions $f_1,\ldots,f_N$ are generated at random and are drawn from a Gaussian process (GP) (see~\cref{sec_GP}).  To keep our theoretical statements manageable, we restrict our attention to PDEs of the form: 
\begin{equation} 
\mathcal{L}u = -\nabla \cdot \left(A(x) \nabla u\right), \quad x\in D, \quad u|_{\partial D} = 0.
\label{eq:PDEsimple}
\end{equation}
Lower-order terms in~\cref{eq:PDEForm} should cause few theoretical problems~\cite{bebendorf2003existence}, though our algorithm and our bounds get far more complicated. 

The approach that dominates the PDE learning literature is to directly learn $\mathcal{L}$ by either (1) Learning parameters in the PDE~\cite{bonito2017diffusion,zhao2020learning}, (2) Using neural networks to approximate the action of the PDE on functions~\cite{raissi2018deep,raissi2018hidden,Karniadakis3,raissi2019physics,raissi2020hidden}, or (3) Deriving a model by composing a library of operators with sparsity considerations~~\cite{Brunton,maddu2019stability,Rudy,schaeffer2017learning,voss2004nonlinear,wang2019variational}.  Instead of trying to learn the unbounded, closed operator $\mathcal{L}$ directly, we follow~\cite{boulle2021data,feliu2020meta,gin2020deepgreen} and discover the Green's function associated with $\mathcal{L}$. That is, we attempt to learn the function $G:D\times D\rightarrow\mathbb{R}^+\cup\{\infty\}$ such that~\cite{evans10}
\begin{equation} 
u_j(x) = \int_{D} G(x,y)f_j(y) \d y,\qquad x\in D, \qquad 1\leq j\leq N.
\label{eq:IntegralOperator}
\end{equation} 
Seeking $G$, as opposed to $\mathcal{L}$, has several theoretical benefits:
\begin{enumerate}[leftmargin=*,noitemsep]

\item The integral operator in~\cref{eq:IntegralOperator} is compact~\cite{edmunds2013bounded}, while $\mathcal{L}$ is only closed~\cite{edmunds2018spectral}. This allows $G$ to be rigorously learned by input-output pairs $\{(f_j,u_j)\}_{j=1}^N$, as its range can be approximated by finite-dimensional spaces (see~\cref{th_Green}). 

\item It is known that $G$ has a hierarchical low-rank structure~\cite[Thm.~2.8]{bebendorf2003existence}: for $0<\epsilon<1$, there exists a function $G_k(x,y) = \sum_{j=1}^k g_j(x)h_j(y)$ with $k = \mathcal{O}(\log^4(1/\epsilon))$ such that~\cite[Thm.~2.8]{bebendorf2003existence}
\[
\left\|G - G_k\right\|_{L^2(X\times Y)}\leq \epsilon\left\|G\right\|_{L^2(X\times \hat{Y})},
\]
where $X,Y\subseteq D$ are sufficiently separated domains, and $Y\subseteq\hat{Y}\subseteq D$ denotes a larger domain than $Y$ (see \cref{theo_bebendorf} for the definition). The further apart $X$ and $Y$, the faster the singular values of $G$ decay. Moreover, $G$ also has an off-diagonal decay property~\cite{gruter1982green,kang2010global}:
\[
G(x,y) \leq \frac{c}{\| x - y\|_2}\|G\|_{L^2(D\times D)}, \qquad x\neq y\in D,
\]
where $c$ is a constant independent of $x$ and $y$. Exploiting these structures of $G$ leads to a rigorous algorithm for constructing a global approximant to $G$ (see~\cref{sec_approx_Green}). 

\item The function $G$ is smooth away from its diagonal, allowing one to efficiently approximate it~\cite{gruter1982green}.

\end{enumerate} 
Once a global approximation $\tilde{G}$ has been constructed for $G$ using input-output pairs, given a new righthand side $f$ one can directly compute the integral in~\cref{eq:IntegralOperator} to obtain the corresponding solution $u$  to~\cref{eq:PDEForm}. Usually, numerically computing the integral in~\cref{eq:IntegralOperator} must be done with sufficient care as $G$ possesses a singularity when $x = y$. However, our global approximation $\tilde{G}$ has an hierarchical structure and is constructed as $0$ near the diagonal. Therefore, for each fixed $x\in D$, we simply recommend that $\int_{D} \tilde{G}(x,y) f_j(y)\d y$ is partitioned into the panels that corresponds to the hierarchical decomposition, and then discretized each panel with a quadrature rule.

\subsection{Main contributions} \label{sec_main_contrib}
There are two main contributions in this paper: (1) The generalization of the randomized singular value decomposition (SVD) algorithm for learning matrices from matrix-vector products to Hilbert--Schmidt (HS) operators and (2) A theoretical learning rate for discovering Green's functions associated with PDEs of the form~\cref{eq:PDEsimple}. These contributions are summarized in \cref{th_tropp_random_svd_Frob,th_Green}.

\cref{th_tropp_random_svd_Frob} says that, with high probability, one can recover a near-best rank $k$ HS operator using $k+p$ operator-function products, for a small integer $p$. In the bound of the theorem, a quantity, denoted by $0< \gamma_k \leq 1$, measures the quality of the input-output training pairs (see~\cref{sec:TwoWarnings,sec_quality_kernel}). We then combine~\cref{th_tropp_random_svd_Frob} with the theory of Green's functions for elliptic PDEs to derive a theoretical learning rate for PDEs.

In \cref{th_Green}, we show that Green's functions associated with uniformly elliptic PDEs in three dimensions can be recovered using $N=\mathcal{O}(\epsilon^{-6}\log^4(1/\epsilon))$ input-output pairs $(f_j,u_j)_{j=1}^N$  to within an accuracy of $\mathcal{O}(\Gamma_\epsilon^{-1/2}\log^3(1/\epsilon)\epsilon)$ with high probability, for $0<\epsilon<1$. Our learning rate associated with uniformly elliptic PDEs in three variables is therefore $\mathcal{O}(\epsilon^{-6}\log^4(1/\epsilon))$. The quantity $0< \Gamma_\epsilon\leq 1$ (defined in~\cref{sec_green_reconst_adm}) measures the quality of the GP used to generate the random functions $\{f_j\}_{j=1}^N$ for learning $G$. We emphasize that the number of training pairs is small only if the GP's quality is high. The probability bound in \cref{th_Green} implies that the constructed approximation is close to $G$ with high probability and converges almost surely to the Green's function as $\epsilon\to 0$.

\subsection{Organization of paper}
The paper is structured as follows. In \cref{sec:background}, we briefly review HS operators and GPs. We then generalize the randomized SVD algorithm to HS operators in~\cref{sec_random_SVD}. Next, in~\cref{sec_approx_Green}, we characterize the learning rate for PDEs of the form of~\cref{eq:PDEsimple} (see~\cref{th_Green}). Finally, we conclude and discuss potential further directions in~\cref{sec_further_work}.

\section{Background material}\label{sec:background} 
We begin by reviewing quasimatrices (see~\cref{sec_Quasimatrices}), HS operators (see~\cref{sec_HS}), and GPs (see~\cref{sec_GP}).  
\subsection{Quasimatrices}\label{sec_Quasimatrices} 
Quasimatrices are an infinite dimensional analogue of tall-skinny matrices~\cite{townsend2015continuous}.  Let $D_1,D_2\subseteq\R^d$ be two domains with $d\geq 1$ and denote by $L^2(D_{1})$ the space of square-integrable functions defined on $D_{1}$.   Many of results in this paper are easier to state using quasimatrices. We say that $\mtx{\Omega}$ is a $D_1\times k$ quasimatrix, if $\mtx{\Omega}$ is a matrix with $k$ columns where each column is a function in $L^2(D_1)$. That is,
\[
\mtx{\Omega} = \begin{bmatrix} \omega_1 \, | & \! \cdots \! & | \, \omega_k \end{bmatrix}, \qquad \omega_j\in L^2(D_1).
\]
Quasimatrices are useful to define analogues of matrix operations for HS operators~\cite{de1991alternative,stewart1998matrix,townsend2015continuous,trefethen1997numerical}. For example, if $\F:L^2(D_1)\to L^2(D_2)$ is a HS operator, then we write $\F\mtx{\Omega}$ to denote the quasimatrix obtained by applying $\F$ to each column of $\mtx{\Omega}$. Moreover, we write $\mtx{\Omega}^*\mtx{\Omega}$ and $\mtx{\Omega}\mtx{\Omega}^*$ to mean the following:
\[
\mtx{\Omega}^*\mtx{\Omega} = \begin{bmatrix}\langle \omega_1,\omega_1 \rangle & \cdots & \langle \omega_1,\omega_k \rangle\\ \vdots & \ddots &  \vdots\\
\langle \omega_k,\omega_1 \rangle & \cdots & \langle \omega_k,\omega_k \rangle \end{bmatrix}, \qquad \mtx{\Omega}\mtx{\Omega}^* = \sum_{j=1}^k \omega_j(x)\omega_j(y),
\]
where $\langle \cdot, \cdot \rangle$ is the $L^2(D_1)$ inner-product.  Many operations for rectangular matrices in linear algebra can be generalized to quasimatrices~\cite{townsend2015continuous}.

\subsection{Hilbert--Schmidt operators} \label{sec_HS}
HS operators are an infinite dimensional analogue of matrices acting on vectors. Since $L^2(D_{1})$ is a separable Hilbert space, there is a complete orthonormal basis $\{e_j\}_{j=1}^\infty$ for $L^2(D_{1})$. We call $\F: L^2(D_1)\to L^2(D_2)$ a HS operator~\cite[Ch.~4]{hsing2015theoretical} with HS norm $\|\F\|_{\HS}$ if $\F$ is linear and
\[
\|\F\|_{\HS} \coloneq \left(\sum_{j=1}^\infty \|\F e_j\|_{L^2(D_2)}^2\right)^{1/2}<\infty.
\]
The archetypical example of an HS operator is an HS integral operator $\F:L^2(D_1)\to L^2(D_2)$ defined by
\[
(\F f)(x)=\int_{D_1} G(x,y)f(y)\d y, \qquad f\in L^2(D_1),\quad x\in D_2,
\]
where $G\in L^2(D_2\times D_1)$ is the kernel of $\F$ and $\|\F\|_{\HS}=\|G\|_{L^2(D_2\times D_1)}$. Since HS operators are compact operators, they have an SVD~\cite[Thm.~4.3.1]{hsing2015theoretical}. That is, 
there exists a nonnegative sequence $\sigma_1\geq \sigma_2\geq \cdots\geq 0$ and an orthonormal basis $\{q_{j}\}_{j=1}^\infty$ for $L^2(D_2)$ such that for any $f\in L^2(D_1)$ we have
\begin{equation} 
\mathscr{F}f = \sum_{\substack{j=1\\\sigma_j>0}}^{\infty}\sigma_j\langle e_j, f\rangle q_{j},
\label{eq:HS_SVD} 
\end{equation} 
where the equality holds in the $L^2(D_2)$ sense. Note that we use the complete SVD, which includes singular functions associated with the kernel of $\F$. Moreover, one finds that $\|\F\|_{\HS}^2 = \sum_{j=1}^\infty \sigma_j^2$, which shows that the HS norm is an infinite dimensional analogue of the Frobenius matrix norm $\|\cdot \|_{\textup{F}}$. In the same way that truncating the SVD after $k$ terms gives a best rank $k$ matrix approximation, truncating~\cref{eq:HS_SVD} gives a best approximation in the HS norm. That is,~\cite[Thm.~4.4.7]{hsing2015theoretical}
\[
\|\F-\F_k\|_{\HS}^2=\sum_{j=k+1}^\infty \sigma_j^2, \qquad \F_kf = \sum_{j=1}^{k}\sigma_j\langle e_j,f\rangle q_{j}, \quad f\in L^2(D_1).
\]
In this paper, we are interested in constructing an approximation to $G$ in~\cref{eq:IntegralOperator} from input-output pairs $\{(f_j,u_j)\}_{j=1}^N$ such that $u_j = \F f_j$.

Throughout this paper, the HS operator denoted by $\mtx{\Omega}\mtx{\Omega}^*\F : L^2(D_1)\to L^2(D_2)$ is given by $\mtx{\Omega}\mtx{\Omega}^*\F f = \sum_{j=1}^k \langle \omega_j,\F f\rangle \omega_j$. If we consider the operator $\mtx{\Omega}^*\F: L^2(D_1)\rightarrow \mathbb{R}^k$, then $\|\mtx{\Omega}^*\F\|_{\HS}^2 = \sum_{j=1}^\infty \|\F e_j\|_{2}^2$. Similarly, for $\F\mtx{\Omega}: \mathbb{R}^k\rightarrow L^2(D_2)$ we have $\|\F\mtx{\Omega}\|_{\HS}^2 = \sum_{j=1}^k \|\F \tilde{e}_j\|_{L^2(D_2)}^2$, where $\{\tilde{e}_j\}_{j=1}^k$ is an orthonormal basis of $\R^k$.   Moreover, if $\mtx{\Omega}$ has full column rank then $\mtx{P}_{\mtx{\Omega}}\F = \mtx{\Omega}(\mtx{\Omega}^*\mtx{\Omega})^{\dagger}\mtx{\Omega}^*\F$ is the orthogonal projection of the range of $\F$ onto the column space of $\mtx{\Omega}$. Here, $(\mtx{\Omega}^*\mtx{\Omega})^{\dagger}$ is the pseudo-inverse of $\mtx{\Omega}^*\mtx{\Omega}$.

\subsection{Gaussian processes} \label{sec_GP}
A GP is an infinite dimensional analogue of a multivariate Gaussian distribution and a function drawn from a GP is analogous to a randomly generated vector. If $K: D\times D\to\R$ is a continuous symmetric positive semidefinite kernel, where $D\subseteq\mathbb{R}^d$ is a domain, then a GP is a stochastic process $\{X_t,\,t\geq 0\}$ such that for every finite set of indices $t_1,\ldots,t_n\geq 0$ the vector of random variables $(X_{t_1},\ldots,X_{t_n})$ is a multivariate Gaussian distribution with mean $(0,\ldots,0)$ and covariance $K_{ij} = K(t_i,t_j)$ for $1\leq i,j\leq n$. We denote a GP with mean $(0,\ldots,0)$ and covariance $K$ by $\mathcal{GP}(0,K)$.

\begin{figure}[htbp]
\vspace{0.5cm}
\begin{center}
\begin{overpic}[width= \textwidth]{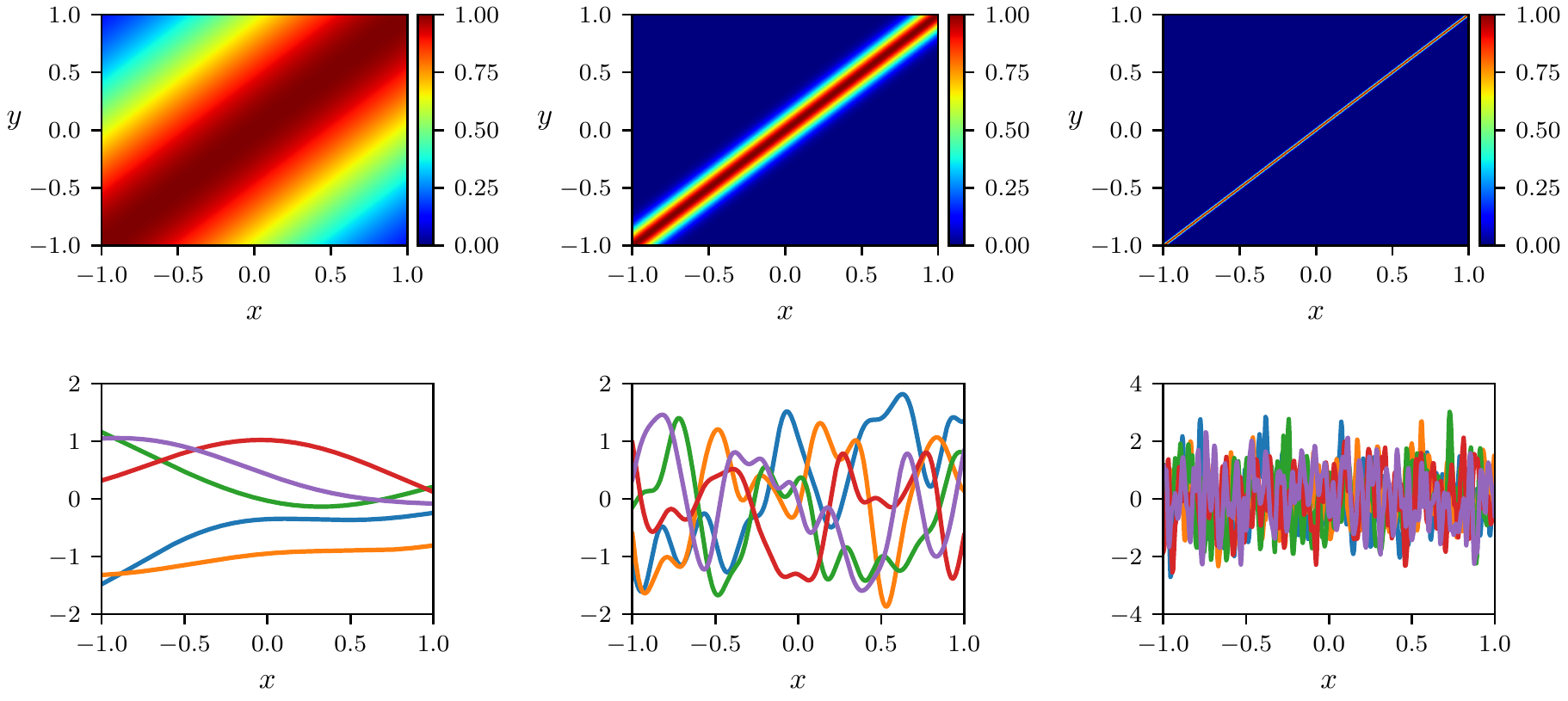}
\put(14,46){$\ell=1$}
\put(48,46){$\ell=0.1$}
\put(81,46){$\ell=0.01$}
\end{overpic}
\end{center}
\caption{Squared-exponential covariance kernel $K_{\text{SE}}$ with parameter $\ell=1,0.1,0.01$ (top row) and five functions sampled from $\mathcal{GP}(0, K_{\text{SE}})$ (bottom row).}
\label{fig_examples_GP}
\end{figure}

Since $K$ is a continuous symmetric positive semidefinite kernel, it has nonnegative eigenvalues $\lambda_1\geq \lambda_2\geq \cdots \geq 0$ and there is an orthonormal basis of eigenfunctions $\{\psi_j\}_{j=1}^\infty$ of $L^2(D)$ such that~\cite[Thm.~4.6.5]{hsing2015theoretical}:
\begin{equation} \label{eq_covariance_kernel}
K(x,y)=\sum_{j=1}^\infty\lambda_j\psi_j(x)\psi_j(y),\qquad \int_{D} K(x,y)\psi_j(y)\d y= \lambda_j \psi_j(x),\qquad x,y\in D,
\end{equation}
where the infinite sum is absolutely and uniformly convergent~\cite{mercer1909fun}. In addition, we define the trace of the covariance kernel $K$ by $\smash{\Tr(K)\coloneq \sum_{j=1}^\infty\lambda_j}<\infty$. The eigendecomposition of $K$ gives an algorithm for generating functions from $\smash{\mathcal{GP}(0,K)}$.  In particular, if $\omega\sim\sum_{j=1}^{\infty} \sqrt{\lambda_j} c_j\psi_j$, where the coefficients $\{c_j\}_{j=1}^\infty$ are independent and identically distributed standard Gaussian random variables, then $\omega\sim\mathcal{GP}(0,K)$~\cite{karhunen1946lineare,loeve1946functions}. We also have
\[
\E\!\left[\|\omega\|_{L^2(D)}^2\right]=\sum_{j=1}^\infty\lambda_j\E\!\left[c_j^2\right]\|\psi_j\|_{L^2(D)}^2=\sum_{j=1}^\infty \lambda_j=\int_{D}K(y,y)\,\d y<\infty,
\]
where the last equality is analogous to the fact that the trace of a matrix is equal to the sum of its eigenvalues. In this paper, we restrict our attention to GPs with positive definite covariance kernels so that the eigenvalues of $K$ are strictly positive. 

In \cref{fig_examples_GP}, we display the squared-exponential kernel defined as $K_{\text{SE}}(x,y)=\exp(-|x-y|^2/(2\ell^2))$ for $x,y\in[-1,1]$~\cite[Chapt.~4]{rasmussen2006gaussian} with parameters $\ell = 1,0.1,0.01$ together with sampled functions from $\mathcal{GP}(0,K_{\text{SE}})$. We observe that the functions become more oscillatory as the length-scale parameter $\ell$ decreases and hence the numerical rank of the kernel increases or, equivalently, the associated eigenvalues $\{\lambda_j\}$ decay more slowly to zero.

\section{Low-rank approximation of Hilbert--Schmidt operators} \label{sec_random_SVD}
In a landmark paper, Halko, Martinsson, and Tropp proved that one could learn the column space of a finite matrix---to high accuracy and with a high probability of success---by using matrix-vector products with standard Gaussian random vectors~\cite{halko2011finding}. We now set out to generalize this from matrices to HS operators. Alternative randomized low-rank approximation techniques such as the generalized Nystr\"om method~\cite{nakatsukasa2020fast} might also be generalized in a similar manner.
Since the proof is relatively long, we state our final generalization now.

\begin{theorem} \label{th_tropp_random_svd_Frob}
Let $D_1,D_2\subseteq \mathbb{R}^d$ be domains with $d\geq 1$ and $\F:L^2(D_1)\to L^2(D_2)$ be a HS operator. Select a target rank $k\geq 1$, an oversampling parameter $p\geq 2$, and a $D_1\times (k+p)$ quasimatrix $\mtx{\Omega}$ such that each column is drawn from $\mathcal{GP}(0,K)$, where $K:D_1\times D_1\to\R$ is a continuous symmetric positive definite kernel with eigenvalues $\lambda_1\geq \lambda_2\geq \cdots>0$. If $\mtx{Y}=\F\mtx{\Omega}$, then
\begin{equation}
\label{eq:MainExpectationBound}
\mathbb{E}\!\left[\|\F - \mtx{P}_\mtx{Y}\F\|_{\HS}\right]\leq \left(1+\sqrt{\frac{1}{\gamma_k}\frac{k(k+p)}{p-1}}\,\right)\left(\sum_{j=k+1}^{\infty}\sigma_j^2\right)^{1/2},
\end{equation}
where $\gamma_k = k/(\lambda_1\Tr(\mtx{C}^{-1}))$ with $\mtx{C}_{ij}=\int_{D_1\times D_1}v_i(x)K(x,y)v_j(y)\d x\d y$ for $1\leq i,j\leq k$. Here, $\mtx{P}_\mtx{Y}$ is the orthogonal projection onto the vector space spanned by the columns of $\mtx{Y}$, $\sigma_j$ is the $j$th singular value of $\F$, and $v_j$ is the $j$th right singular vector of $\F$. 

Assume further that $p\geq 4$, then for any $s,t\geq 1$, we have
\begin{equation}
\label{eq:MainProbabilityBound}
\|\F-\mtx{P}_{\mtx{Y}}\F\|_{\HS}\leq \sqrt{1+ t^2s^2 \frac{3}{\gamma_k}\frac{k(k+p)}{p+1}\sum_{j=1}^\infty \frac{\lambda_j}{\lambda_1}}\,\left(\sum_{j=k+1}^\infty\sigma_j^2\right)^{1/2},
\end{equation}
with probability $\geq 1 - t^{-p}-[s e^{-(s^2-1)/2}]^{k+p}$.
\end{theorem}

We remark that the term $[s e^{-(s^2-1)/2}]^{k+p}$ in the statement of \cref{th_tropp_random_svd_Frob} is bounded by $e^{-s^2}$ for $s\geq 2$ and $k+p\geq 5$. In the rest of the section, we prove this theorem.

\subsection{Three caveats that make the generalization non-trivial}\label{sec:TwoWarnings}
One might imagine that the generalization of the randomized SVD algorithm from matrices to HS operators is trivial, but this is not the case due to three caveats: 

\begin{enumerate}[leftmargin=*,noitemsep]
\item The randomized SVD on finite matrices always uses matrix-vector products with standard Gaussian random vectors~\cite{halko2011finding}. However, for GPs, one must always have a continuous kernel $K$ in $\mathcal{GP}(0,K)$, which discretizes to a non-standard multivariate Gaussian distribution. Therefore, we must extend~\cite[Thm.~10.5]{halko2011finding} to allow for non-standard multivariate Gaussian distributions.  The discrete version of our extension is the following:

\begin{corollary} \label{cor_finite_dim}
Let $\mtx{A}$ be a real $n_2\times n_1$ matrix with singular values $\sigma_1\geq \cdots \geq \sigma_{\min\{n_1,n_2\}}$. Choose a target rank $k\geq 1$ and an oversampling parameter $p\geq 2$. Draw an $n_1\times (k+p)$ Gaussian matrix, $\mtx{\Omega}$, with independent columns where each column is from a multivariate Gaussian distribution with mean $(0,\ldots,0)^\top$ and positive definite covariance matrix $\mtx{K}$. If $\mtx{Y}=\mtx{A}\mtx{\Omega}$, then the expected approximation error is bounded by
\begin{equation} \label{eq_expect_cor}
\mathbb{E}\left[\|\mtx{A}-\mtx{P}_\mtx{Y}\mtx{A}\|_{\textup{F}}\right]\leq \left(1+\sqrt{\frac{k+p}{p-1}\sum_{j=n_1-k+1}^{n_1}\frac{\lambda_1}{\lambda_j}}\,\right)\left(\sum_{j=k+1}^\infty \sigma_j^2\right)^{1/2},
\end{equation}
where $\lambda_1 \geq \cdots \geq \lambda_{n_1} > 0$ are the eigenvalues of $\mtx{K}$ and $\mtx{P}_\mtx{Y}$ is the orthogonal projection onto the vector space spanned by the columns of $\mtx{Y}$. Assume further that $p\geq 4$, then for any $s,t\geq 1$, we have
\[\|\mtx{A}-\mtx{P}_\mtx{Y}\mtx{A}\|_{\textup{F}}\leq \left(\!1+ ts\cdot \sqrt{\frac{3(k+p)}{p+1}\left(\sum_{j=1}^{n_1}\lambda_j\right)\sum_{j=n_1-k+1}^{n_1}\frac{1}{\lambda_j}}\,\right)\!\left(\sum_{j=k+1}^\infty\sigma_j^2\right)^{1/2},\]
with probability $\geq 1 - t^{-p}-[s e^{-(s^2-1)/2}]^{k+p}$.
\end{corollary}

Choosing a covariance matrix $\mtx{K}$ with sufficient eigenvalue decay so that $\lim_{n_1\rightarrow\infty}\sum_{j=1}^{n_1} \lambda_j <\infty$ allows $\E[\|\mtx{\Omega}\|_{\textup{F}}^2]$ to remain bounded as $n_1\to\infty$. This is of interest when applying the randomized SVD algorithm to extremely large matrices and is critical for HS operators. A stronger statement of this result~\cite[Thm.~2]{boulle2021generalization} shows that prior information on $\mtx{A}$ can be incorporated into the covariance matrix to achieve lower approximation error than the randomized SVD with standard Gaussian vectors.

\item We need an additional essential assumption. The kernel in $\mathcal{GP}(0,K)$ is ``reasonable" for learning $\F$, where reasonableness is measured by the quantity $\gamma_k$ in \cref{th_tropp_random_svd_Frob}. If the first $k$ right singular functions of the HS operator $v_1,\ldots,v_k$ are spanned by the first $k+m$ eigenfunctions of $K$ $\psi_1,\ldots,\psi_{k+m}$, for some $m\in \mathbb{N}$, then (see~\cref{eq_lower_bound_gamma} and~\cref{lem_bound_sigma})
\[
\frac{1}{k}\sum_{j=1}^k\frac{\lambda_1}{\lambda_j}\leq \frac{1}{\gamma_k} \leq \frac{1}{k}\sum_{j=m+1}^{k+m}\frac{\lambda_1}{\lambda_j}.
\] 
In the matrix setting, this assumption always holds with $m=n_1-k$ (see~\cref{cor_finite_dim}) and one can have $\gamma_k = 1$ when $\lambda_1=\cdots=\lambda_{n_1}$~\cite[Thm.~10.5]{halko2011finding}.

\item Probabilistic error bounds for the randomized SVD in~\cite{halko2011finding} are derived using tail bounds for functions of standard Gaussian matrices~\cite[Sec.~5.1]{ledoux2001concentration}. Unfortunately, we are not aware of tail bounds for non-standard Gaussian quasimatrices. This results in a slightly weaker probability bound than~\cite[Thm.~10.7]{halko2011finding}.
\end{enumerate} 

\subsection{Deterministic error bound}
Apart from the three caveats, the proof of \cref{th_tropp_random_svd_Frob} follows the outline of the argument in~\cite[Thm.~10.5]{halko2011finding}. We define two quasimatrices $\mtx{U}$ and $\mtx{V}$ containing the left and right singular functions of $\F$ so that the $j$th column of $\mtx{V}$ is $v_j$. We also denote by $\mtx{\Sigma}$ the infinite diagonal matrix with the singular values of $\F$, i.e., $\sigma_1\geq\sigma_2\geq\cdots\geq0$, on the diagonal. Finally, for a fixed $k\geq 1$, we define the $D_1\times k$ quasimatrix as the truncation of $\mtx{V}$ after the first $k$ columns and $\mtx{V}_2$ as the remainder. Similarly, we split $\mtx{\Sigma}$ into two parts:
\[\begin{array}{@{}c@{}c@{}c@{}c@{}c@{}c}
        && k & \infty &\\
    \mtx{\Sigma} =  &\left. \begin{array}{c} \\ \\ \end{array} \!\!\! \right( &
    \begin{array}{c} \mtx{\Sigma}_1 \\ 0 \end{array} &
    \begin{array}{c} 0 \\ \mtx{\Sigma}_2 \end{array} &
    \left. \!\!\!\begin{array}{c} \\ \\ \end{array}  \right) &
    \begin{array}{c} k \\ \infty \\ \end{array}
\end{array}.\]
We are ready to prove an infinite dimensional analogue of~\cite[Thm.~9.1]{halko2011finding} for HS operators.

\begin{theorem}[Deterministic error bound] \label{th_tropp_deter_svd}
Let $\F:L^2(D_1)\to L^2(D_2)$ be a HS operator with SVD given in~\cref{eq:HS_SVD}. Let $\mtx{\Omega}$ be a $D_1\times \ell$ quasimatrix and $\mtx{Y}=\F\mtx{\Omega}$. If $\mtx{\Omega}_1=\mtx{V}_1^*\mtx{\Omega}$ and $\mtx{\Omega}_2=\mtx{V}_2^*\mtx{\Omega}$, then assuming $\mtx{\Omega}_1$ has full rank, we have
\[
\|\F-\mtx{P}_\mtx{Y}\F\|_{\HS}^2\leq \|\mtx{\Sigma}_2\|_{\HS}^2+\|\mtx{\Sigma}_2\mtx{\Omega}_2\mtx{\Omega}_1^{\dagger}\|_{\HS}^2,
\]
where $\mtx{P}_\mtx{Y}=\mtx{Y}(\mtx{Y}^*\mtx{Y})^\dagger\mtx{Y}^*$ is the orthogonal projection onto the space spanned by the columns of $\mtx{Y}$ and $\smash{\mtx{\Omega}_1^{\dagger} = (\mtx{\Omega}_1^*\mtx{\Omega}_1)^{-1}\mtx{\Omega}_1^*}$. 
\end{theorem}
\begin{proof}
First, note that because $\mtx{U}\mtx{U}^*$ is the orthonormal projection onto the range of $\F$ and $\mtx{U}$ is a basis for the range, we have
\[
\|\F-\mtx{P}_\mtx{Y}\F\|_{\HS} = \|\mtx{U}\mtx{U}^*\F-\mtx{P}_{\mtx{Y}}\mtx{U}\mtx{U}^*\F\|_{\HS}.
\]
By Parseval's theorem~\cite[Thm.~4.18]{rudin1987real}, we have
\[ 
\|\mtx{U}\mtx{U}^* \F-\mtx{P}_{\mtx{Y}}\mtx{U}\mtx{U}^* \F\|_{\HS}= \|\mtx{U}^*\mtx{U}\mtx{U}^* \F-\mtx{U}^*\mtx{P}_{\mtx{Y}}\mtx{U}\mtx{U}^* \F \mtx{V}\|_{\HS}.
\]
Moreover, we have the equality $\|\F-\mtx{P}_{\mtx{Y}}\F\|_{\HS}=\|(\mtx{I}-\mtx{P}_{\mtx{U}^*\mtx{Y}})\mtx{U}^*\F\mtx{V}\|_{\HS}$ 
because the inner product $\langle\sum_{j=1}^\infty \alpha_j u_j,\sum_{j=1}^\infty \beta u_j\rangle=0$ if and only if $\sum_{j=1}^\infty \alpha_j \beta_j=0$. We now take $\mtx{A}=\mtx{U}^*\F\mtx{V}$, which is a bounded infinite matrix such that $\|\mtx{A}\|_{\textup{F}}=\|\F\|_{\HS}<\infty$. The statement of the theorem immediately follows from the proof of~\cite[Thm.~9.1]{halko2011finding}.
\end{proof}

This theorem shows that the bound on the approximation error $\|\F-\mtx{P}_\mtx{Y}\F\|_{\HS}$ depends on the singular values of the HS operator and the test matrix $\mtx{\Omega}$.

\subsection{Probability distribution of $\mtx{\Omega}_1$} \label{sec_proba_omega_1}

If the columns of $\mtx{\Omega}$ are independent and identically distributed as $\mathcal{GP}(0,K)$, then the matrix $\mtx{\Omega}_1$ in~\cref{th_tropp_deter_svd} is of size $k\times \ell$ with entries that follow a Gaussian distribution.  To see this, note that 
\[
\mtx{\Omega}_1 = \mtx{V}_1^*\mtx{\Omega} = 
\left(
\begin{array}{ccc}
\langle v_1,\omega_1\rangle & \cdots & \langle v_1, \omega_{\ell}\rangle\\
\vdots & \ddots & \vdots \\
\langle v_k,\omega_1\rangle & \cdots & \langle v_k,\omega_{\ell}\rangle
\end{array}
\right), \qquad \omega_j\sim\mathcal{GP}(0,K). 
\]
If $\omega \sim \mathcal{GP}(0,K)$ with $K$ given in \cref{eq_covariance_kernel}, then we find that $\langle v,\omega\rangle \sim \mathcal{N}(0,\sum_{j=1}^\infty \lambda_j \langle v,\psi_j\rangle^2)$ so we conclude that $\mtx{\Omega}_1$ has Gaussian entries with zero mean. Finding the covariances between the entries is more involved. 
\begin{lemma} \label{prop_cov_matrix}
With the same setup as~\cref{th_tropp_deter_svd}, suppose that the columns of $\mtx{\Omega}$ are independent and identically distributed as $\mathcal{GP}(0,K)$. Then, the matrix $\mtx{\Omega}_1=\mtx{V}_1^*\mtx{\Omega}$ in~\cref{th_tropp_deter_svd} has independent columns and each column is identically distributed as a multivariate Gaussian with positive definite covariance matrix $\mtx{C}$ given by
\begin{equation} \label{eq_def_sigma_k}
\mtx{C}_{ij} = \int_{D_1\times D_1} v_i(x)K(x,y)v_{j}(y)\d x\d y, \qquad 1\leq i,j\leq k,
\end{equation}
where $v_i$ is the $i$th column of $\mtx{V}_1$.
\end{lemma}
\begin{proof}
We already know that the entries are Gaussian with mean $0$. Moreover, the columns are independent because $\omega_1,\ldots,\omega_\ell$ are independent. Therefore, we focus on the covariance matrix. Let $1\leq i,i'\leq k$, $1\leq j,j'\leq \ell$, then since $\mathbb{E}\!\left[\langle v_i,\omega_j\rangle\right] = 0$ we have
\[
\cov(\langle v_i,\omega_j\rangle, \langle v_{i'},\omega_{j'}\rangle) 
= \E\left[ \langle v_i,\omega_j\rangle\, \langle v_{i'},\omega_{j'}\rangle \right] = \mathbb{E}\left[X_{ij}X_{i'j'}\right],
\]
where $X_{ij} = \langle v_i,\omega_j\rangle$. Since $\langle v_i,\omega_j\rangle\sim\sum_{n=1}^\infty \sqrt{\lambda_n}c_n^{(j)} \langle v_i,\psi_n\rangle$, where $c_n^{(j)}\sim\mathcal{N}(0,1)$, we have
\[
\cov(\langle v_i,\omega_j\rangle, \langle v_{i'},\omega_{j'}\rangle) = \E\left[\lim_{m_1,m_2\to\infty}X_{ij}^{m_1} X_{i'j'}^{m_2}\right],\qquad X_{ij}^{m_1}\coloneq\sum_{n=1}^{m_1} \sqrt{\lambda_n}c_n^{(j)}  \langle v_i,\psi_n\rangle.
\]
We first show that $\lim_{m_1,m_2\to\infty}\left|\E\!\left[\!X_{ij}^{m_1}X_{i'j'}^{m_2}\right]-\E\!\left[X_{ij}X_{i'j'}\right]\right|=0$. For any $m_1,m_2\geq 1$, we have by the triangle inequality,
\begin{align*}
\left|\E\!\left[\!X_{ij}^{m_1}X_{i'j'}^{m_2}\!\right]-\E\!\left[X_{ij}X_{i'j'}\right]\right| \!\!&\leq \E\!\left[\left|X_{ij}^{m_1}X_{i'j'}^{m_2}-X_{ij}X_{i'j'}\right|\right]\\
&\leq \E\!\left[\left|(X_{ij}^{m_1}-X_{ij})X_{i'j'}^{m_2}\right|\right] \!\!+\E\!\left[\left|X_{ij}(X_{i'j'}^{m_2}-X_{i'j'})\right|\right]\\
&\!\!\!\!\!\!\!\!\!\!\!\!\!\!\!\!\!\!\!\!\!\!\!\!\!\!\!\!\!\!\!\!\!\!\!\!
\leq 
\E\!\left[\left|X_{ij}^{m_1}-X_{ij}\right|^2\right]^{\tfrac{1}{2}}\!\E\!\left[\left|X_{i'j'}^{m_2}\right|^2\right]^{\tfrac{1}{2}}\!\!
+
\E\!\left[\left|X_{i'j'}-X_{i'j'}^{m_2}\right|^2\right]^{\tfrac{1}{2}}\!\E\!\left[\left|X_{ij}\right|^2\right]^{\tfrac{1}{2}}\!,
\end{align*}
where the last inequality follows from the Cauchy--Schwarz inequality. We now set out to show that both terms in the last inequality converge to zero as $m_1,m_2\to\infty$. The terms $\smash{\E[|X_{i'j'}^{m_2}|^2]}$ and $\smash{\E[|X_{ij}|^2]}$ are bounded by $\sum_{n=1}^\infty \lambda_n<\infty$, using the Cauchy--Schwarz inequality. Moreover, we have
\[
\E\left[\left|X_{ij}^{m_1}-X_{ij}\right|^2\right] = \E\left[\left|\sum_{n=m_1+1}^\infty \sqrt{\lambda_n}c_n^{(j)} \langle v_i,\psi_n\rangle\right|^2\right]\leq \sum_{n=m_1+1}^\infty \lambda_n \xrightarrow[m_1 \to \infty]{} 0,
\]
because $X_{ij} - X_{ij}^{m_1}\sim \mathcal{N}(0,\sum_{n=m_1+1}^\infty\lambda_n\langle v_i,\psi_n\rangle^2)$. Therefore, we find that $\cov(X_{ij}, X_{i'j'}) = \lim_{m_1,m_2\to\infty}\E[X_{ij}^{m_1}X_{i'j'}^{m_2}]$ and we obtain
\begin{align*}
\cov(X_{ij},X_{i'j'})
& =\lim_{m_1,m_2\to\infty}\E\left[\sum_{n=1}^{m_1}\sum_{n'=1}^{m_2} \sqrt{\lambda_n \lambda_{n'}}c_n^{(j)}c_{n'}^{(j')} \langle v_i,\psi_n\rangle\langle v_{i'},\psi_{n'}\rangle\right] \\
&=\lim_{m_1,m_2\to\infty}\sum_{n=1}^{m_1}\sum_{n'=1}^{m_2} \sqrt{\lambda_n \lambda_{n'}}\E[c_n^{(j)}c_{n'}^{(j')}] \langle v_i,\psi_n\rangle\langle v_{i'},\psi_{n'}\rangle.
\end{align*}
The latter expression is zero if $n\neq n'$ or $j\neq j'$ because then $c_n^{(j)}$ and $c_{n'}^{(j')}$ are independent random variables with mean $0$. Since $\E[(c_n^{(j)})^2] = 1$, we have
\[
\cov(X_{ij},X_{i'j'})= 
\begin{cases}
\sum_{n=1}^\infty \lambda_n \langle v_i,\psi_n\rangle\langle v_{i'},\psi_{n}\rangle,  & j=j',\\
0, & \text{otherwise}.
\end{cases}
\]
The result follows as the infinite sum is equal to the integral in~\cref{eq_def_sigma_k}.  To see that $\mtx{C}$ is positive definite, let $a\in \R^k$, then $a^* \mtx{C} a=\E[Z_a^2]\geq 0$, where $Z_a\sim\mathcal{N}(0,\sum_{n=1}^\infty \lambda_n \langle a_1 v_1+\cdots+a_k v_k,\psi_n\rangle^2)$. Moreover, $a^* \mtx{C} a = 0$ implies that $a=0$ because $v_1,\ldots,v_k$ are orthonormal and $\{\psi_n\}$ is an orthonormal basis of $L^2(D_1)$. 
\end{proof}

\cref{prop_cov_matrix} gives the distribution of the matrix $\mtx{\Omega}_1$, which is essential to prove \cref{th_tropp_random_svd_Frob} in \cref{sec_proof_thm_1}. In particular, $\mtx{\Omega}_1$ has independent columns that are each distributed as a multivariate Gaussian with covariance matrix given in \cref{eq_def_sigma_k}.

\subsection{Quality of the covariance kernel} \label{sec_quality_kernel}
To investigate the quality of the kernel, we introduce the Wishart distribution, which is a family of probability distributions over symmetric and nonnegative-definite matrices that often appear in the context of covariance matrices~\cite{wishart1928generalised}. If $\mtx{\Omega}_1$ is a $k\times \ell$ random matrix with independent columns, where each column is a multivariate Gaussian distribution with mean $(0,\ldots,0)^\top$ and covariance $\mtx{C}$, then $\mtx{A}=\mtx{\Omega}_1 \mtx{\Omega}_1^*$ has a Wishart distribution~\cite{wishart1928generalised}. We write $\mtx{A}\sim W_k(\ell,\mtx{C})$. We note that $\|\mtx{\Omega}_1^{\dagger}\|_{\textup{F}}^2=\Tr[(\mtx{\Omega}_1^{\dagger})^*\mtx{\Omega}_1^{\dagger}]=\Tr(\mtx{A}^{-1})$, where the second equality holds with probability one because the matrix $\mtx{A}=\mtx{\Omega}_1\mtx{\Omega}_1^*$ is invertible with probability one (see~\cite[Thm.~3.1.4]{muirhead2009aspects}). By~\cite[Thm.~3.2.12]{muirhead2009aspects} for $\ell-k\geq 2$, we have $\E[\mtx{A}^{-1}]=\frac{1}{\ell-k-1}\mtx{C}^{-1}$, $\E[\Tr(\mtx{A}^{-1})]=\Tr(\mtx{C}^{-1})/(\ell-k-1)$, and conclude that 
\begin{equation} \label{prop_rand_2}
\mathbb{E}\left[\|\mtx{\Omega}_1^{\dagger}\|_{\textup{F}}^2\right]=\frac{1}{\gamma_k\lambda_1}\frac{k}{\ell-k-1},\qquad \gamma_k \coloneq \frac{k}{\lambda_1 \Tr(\mtx{C}^{-1})}. 
\end{equation} 

The quantity $\gamma_k$ can be viewed as measuring the quality of the covariance kernel $K$ for learning the HS operator $\F$ (see~\cref{th_tropp_random_svd_Frob}). First, $1\leq \gamma_k<\infty$ as $\mtx{C}$ is symmetric positive definite. Moreover, for $1\leq j\leq k$, the $j$th largest eigenvalue of $\mtx{C}$ is bounded by the $j$th largest eigenvalue of $K$ as $\mtx{C}$ is a principal submatrix of $\mtx{V}^*K\mtx{V}$~\cite[Sec.~III.5]{kato2013perturbation}. Therefore, the following inequality holds,
\begin{equation} \label{eq_lower_bound_gamma}
\frac{1}{k}\sum_{j=1}^k\frac{\lambda_1}{\lambda_j}\leq \frac{1}{\gamma_k}<\infty,
\end{equation}
and the harmonic mean of the first $k$ scaled eigenvalues of $K$ is a lower bound for $1/\gamma_k$. In the ideal situation, the eigenfunctions of $K$ are the right singular functions of $\F$, i.e., $\psi_n=v_n$, $\mtx{C}$ is a diagonal matrix with entries $\lambda_1,\ldots,\lambda_k$, and $\gamma_k=k/(\sum_{j=1}^k \lambda_1/\lambda_j)$ is as small as possible.

We now provide a useful upper bound on $\gamma_k$ in a more general setting.

\begin{lemma} \label{lem_bound_sigma}
Let $\mtx{V}_1$ be a $D_1\times k$ quasimatrix with orthonormal columns and assume that there exists $m\in\mathbb{N}$ such that the columns of $\mtx{V}_1$ are spanned by the first $k+m$ eigenvectors of the continuous positive definite kernel $K:D_1\times D_1\to \R$. Then 
\[
\frac{1}{\gamma_k} \leq \frac{1}{k}\sum_{j=m+1}^{k+m}\frac{\lambda_1}{\lambda_j},
\]
where $\lambda_1\geq \lambda_2\geq\cdots> 0$ are the eigenvalues of $K$. This bound is tight in the sense that the inequality can be attained as an equality. 
\end{lemma}
\begin{proof}
Let $\mtx{Q} = \left[v_1\,|\,\cdots\,|\,v_k\,|\,q_{k+1}\,|\cdots\,|\,q_{k+m}\right]$ be a quasimatrix with orthonormal columns whose columns form an orthonormal basis for ${\rm Span}(\psi_1,\ldots,\psi_{k+m})$.  Then, $\mtx{Q}$ is an invariant space of $K$ and $\mtx{C}$ is a principal submatrix of $\mtx{Q}^*K\mtx{Q}$, which has eigenvalues $\lambda_1\geq \cdots\geq \lambda_{k+m}$. By~\cite[Thm.~6.46]{kato2013perturbation} the $k$ eigenvalues of $\mtx{C}$, denoted by $\mu_1,\ldots,\mu_{k}$, are greater than the first $k+m$ eigenvalues of $K$: $\mu_{j}\geq \lambda_{m+j}$ for $1\leq j\leq k$, and the result follows as the trace of a matrix is the sum of its eigenvalues. 
\end{proof}

\subsection{Probabilistic error bounds}
As discussed in~\cref{sec:TwoWarnings}, we need to extend the probability bounds of the randomized SVD to allow for non-standard Gaussian random vectors. The following lemma is a generalization of~\cite[Thm.~A.7]{halko2011finding}.
\begin{lemma} \label{thm_propa_bound}
Let $k,\ell\geq 1$ such that $\ell-k\geq 4$ and $\mtx{\Omega}_1$ be a $k\times \ell$ random matrix with independent columns such that each column has mean $(0,\ldots,0)^\top$ and positive definite covariance $\mtx{C}$. For all $t\geq 1$, we have
\[\mathbb{P}\left\{\|\mtx{\Omega}_1^\dagger\|_{\textup{F}}^2>\frac{3\Tr(\mtx{C}^{-1})}{\ell-k+1}\cdot t^2\right\}\leq t^{-(\ell-k)}.\]
\end{lemma}
\begin{proof}
Since $\mtx{\Omega}_1\mtx{\Omega}_1^*\sim W_{k}(\ell,\mtx{C})$, the reciprocals of its diagonal elements follow a scaled chi-square distribution~\cite[Thm.~3.2.12]{muirhead2009aspects}, i.e.,
\[
\frac{\left((\mtx{\Omega}_1\mtx{\Omega}_1^*)^{-1}\right)_{jj}}{\left(\mtx{C}^{-1}\right)_{jj}} \sim X_j^{-1},\qquad X_j\sim \chi_{\ell-k+1}^2,\qquad 1\leq j\leq k.
\]
Let $Z=\|\mtx{\Omega}_1^\dagger\|_{\textup{F}}^2=\Tr[(\mtx{\Omega}_1\mtx{\Omega}_1^*)^{-1}]$ and $q=(\ell-k)/2$. Following the proof of~\cite[Thm.~A.7]{halko2011finding}, we have the inequality
\[
\mathbb{P}\left\{|Z|\geq \frac{3\Tr(\mtx{C}^{-1})}{\ell-k+1}\cdot t^2\right\}\leq \left[\frac{3\Tr(\mtx{C}^{-1})}{\ell-k+1}\cdot t^2\right]^{-q}\E\left[|Z|^q\right],\quad t\geq 1.
\]
Moreover, by the Minkowski inequality, we have
\[
\left(\E\left[|Z^q|\right]\right)^{1/q}=\left(\E\left[\left|\sum_{j=1}^k[\mtx{C}^{-1}]_{jj}X_j^{-1}\right|^q\right]\right)^{1/q}\!\!\leq 
\sum_{j=1}^k [\mtx{C}^{-1}]_{jj}\E\left[|X_j^{-1}|^q\right]^{1/q}\leq \frac{3\Tr(\mtx{C}^{-1})}{\ell-k+1},
\]
where the last inequality is from~\cite[Lem.~A.9]{halko2011finding}. The result follows from the argument in the proof of~\cite[Thm.~A.7]{halko2011finding}.
\end{proof}

Under the assumption of \cref{lem_bound_sigma}, we find that~\cref{thm_propa_bound} gives the following bound: 
\[
\mathbb{P}\left\{\|\mtx{\Omega}_1^\dagger\|_{\textup{F}}>t\cdot \sqrt{\frac{3}{\ell-k+1}\sum_{j=m+1}^{k+m}\lambda_j^{-1}}\right\}\leq t^{-(\ell-k)}.
\]
In particular, in the finite dimensional case when $\lambda_1=\cdots=\lambda_n=1$, we recover the probabilistic bound found in~\cite[Thm.~A.7]{halko2011finding}.

To obtain the probability statement found in \cref{eq_proba_Green_bound} we require control of the tail of the distribution of a Gaussian quasimatrix with non-standard covariance kernel (see~\cref{sec_proof_thm_1}). In the theory of the randomized SVD, one relies on the concentration of measure results~\cite[Prop.~10.3]{halko2011finding}. However, we need to employ a different strategy and instead directly bound the HS norm of $\mtx{\Omega}_2$. One difficulty is that the norm of this matrix must be controlled for large dimensions $n$, which leads to a weaker probability bound than~\cite{halko2011finding}. While it is possible to apply Markov's inequality to obtain deviation bounds, we highlight that \cref{prop_control_matrices} provides a Chernoff-type bound, i.e., exponential decay of the tail distribution of $\|\mtx{\Omega}_2\|_{\HS}$, which is crucial to approximate Green's functions (see \cref{sec_proba_fail_Green}).

\begin{lemma} \label{prop_control_matrices}
With the same notation as in~\cref{th_tropp_deter_svd}, let $\ell\geq k \geq 1$. For all $s\geq 1$ we have
\[\mathbb{P}\left\{\|\mtx{\Omega}_2\|_{\HS}^2>\ell s^2\Tr(K)\right\}\leq  \left[s e^{-(s^2-1)/2}\right]^{\ell}.\]
\end{lemma}

\begin{proof}
We first remark that
\begin{equation} \label{eq_distrib_omega}
\|\mtx{\Omega}_2\|_{\HS}^2 \leq \|\mtx{\Omega}\|_{\HS}^2 = \sum_{j=1}^\ell Z_j,\qquad Z_j \coloneq \|\omega_j\|_{L^2(D_1)}^2,
\end{equation}
where the $Z_j$ are independent and identically distributed (i.i.d) because $\omega_j\sim \mathcal{GP}(0,K)$ are i.i.d. For $1\leq j\leq \ell$, we have (c.f.~\cref{sec_GP}),
\[\omega_j = \sum_{m=1}^\infty c_m^{(j)}\sqrt{\lambda_m}\psi_m,\]
where $c_m^{(j)}\sim \mathcal{N}(0,1)$ are i.i.d for $m\geq 1$ and $1\leq j\leq \ell$. First, since the series in \cref{eq_distrib_omega} converges absolutely, we have
\[Z_j = \sum_{m=1}^\infty (c_m^{(j)})^2\lambda_m= \lim_{N\to\infty}\sum_{m=1}^N X_m, \qquad X_m = (c_m^{(j)})^2\lambda_m,\]
where the $X_m$ are independent random variables and $X_m \sim \lambda_m\chi^2$ for $1\leq m\leq N$. Here, $\chi^2$ denotes the chi-squared distribution~\cite[Chapt.~4.3]{mood1950introduction}.

Let $N\geq 1$ and $0<\theta<1/(2\Tr(K))$, we can bound the moment generating function of $\sum_{m=1}^N X_m$ as
\begin{align*}
\E\left[e^{\theta \sum_{m=1}^N X_m}\right] &= \prod_{m=1}^N\E\left[e^{\theta X_m}\right] = \prod_{m=1}^N (1-2\theta\lambda_m)^{-1/2}\leq \left(1-2\theta \sum_{m=1}^N \lambda_m\right)^{-1/2}\\
&\leq \left(1-2\theta \Tr(K)\right)^{-1/2},
\end{align*}
because $X_m/\lambda_m$ are independent random variables that follow a chi-squared distribution. Using the monotone convergence theorem, we have
\[\E\left[e^{\theta Z_j}\right]\leq (1-2\theta \Tr(K))^{-1/2}.\]

Let $\tilde{s}\geq 0$ and $0<\theta<1/(2\Tr(K))$, by the Chernoff bound~\cite[Thm.~1]{chernoff1952measure}, we obtain
\begin{align*}
\mathbb{P} \left\{\|\mtx{\Omega}_2\|_{\HS}^2 > \ell (1+\tilde{s})\Tr(K)\right\} &\leq e^{-(1+\tilde{s})\Tr(K)\ell \theta}\E\left[e^{\theta Z_j}\right]^\ell \\
&=e^{-(1+\tilde{s})\Tr(K)\ell \theta}(1-2\theta \Tr(K))^{-\ell/2}.
\end{align*}
We can minimize this upper bound over $0<\theta<1/(2\Tr(K))$ by choosing $\theta = \tilde{s}/(2(1+\tilde{s})\Tr(K))$, which gives
\[\mathbb{P} \left\{\|\mtx{\Omega}_2\|_{\HS}^2 > \ell (1+\tilde{s})\Tr(K)\right\}\leq (1+\tilde{s})^{\ell/2} e^{-\ell \tilde{s}/2}.\]
Choosing $s=\sqrt{1+\tilde{s}}\geq 1$ concludes the proof.
\end{proof}

\cref{prop_control_matrices} can be refined further to take into account the interaction between the Hilbert--Schmidt operator $\F$ and the covariance kernel $K$ (see~\cite[Lem.~7]{boulle2021generalization}).

\subsection{Randomized SVD algorithm for HS operators}\label{sec_proof_thm_1}

We first prove an intermediary result, which generalizes \cite[Prop.~10.1]{halko2011finding} to HS operators. Note that one may obtain sharper bounds using a suitably chosen covariance kernels that yields a lower approximation error~\cite{boulle2021generalization}.

\begin{lemma} \label{prop_rand_1}
Let $\mtx{\Sigma}_2$, $\mtx{V}_2$, and $\mtx{\Omega}$ be defined as in \cref{th_tropp_deter_svd}, and $\mtx{T}$ be an $\ell\times k$ matrix, where $\ell\geq k\geq 1$. Then,
\[
\E\left[\|\mtx{\Sigma}_2 \mtx{V}_2^* \mtx{\Omega} \mtx{T}\|_{\HS}^2\right] \leq \lambda_1\|\mtx{\Sigma}_2\|_{\HS}^2\|\mtx{T}\|_{\textup{F}}^2,
\]
where $\lambda_1$ is the first eigenvalue of $K$.
\end{lemma}
\begin{proof}
Let $\mtx{T} = \mtx{U}_\mtx{T} \mtx{D}_\mtx{T} \mtx{V}_\mtx{T}^*$ be the SVD of $\mtx{T}$. If $\{v_{\mtx{T},i}\}_{i=1}^k$ are the columns of $\mtx{V}_\mtx{T}$, then 
\[
\E\left[\|\mtx{\Sigma}_2 \mtx{V}_2^* \mtx{\Omega}\mtx{T}\|_{\HS}^2\right]= \sum_{i=1}^k\E\left[\|\mtx{\Sigma}_2 \mtx{\Omega}_2\mtx{U}_{\mtx{T}} \mtx{D}_\mtx{T} \mtx{V}_\mtx{T}^* v_{\mtx{T},i}\|_2^2\right],
\]
where $\mtx{\Omega}_2=\mtx{V}_2^*\mtx{\Omega}$. Therefore, we have
\[
\E\left[\|\mtx{\Sigma}_2 \mtx{\Omega}_2\mtx{T}\|_{\HS}^2\right] = \sum_{i=1}^k ((\mtx{D}_\mtx{T})_{ii})^2 \E\left[\|\mtx{\Sigma}_2 \mtx{\Omega}_2\mtx{U}_{\mtx{T}}(:,i)\|_{2}^2\right].
\]
Moreover, using the monotone convergence theorem for non-negative random variables, we have 
\begin{align*}
\E\left[\|\mtx{\Sigma}_2 \mtx{\Omega}_2\mtx{U}_{\mtx{T}}(:,i)\|_{2}^2\right]
&= \E\left[\sum_{n=1}^\infty\sum_{j=1}^{\ell}\sigma_{k+n}^2 \left|\mtx{\Omega}_2(n,j)\right|^2\mtx{U}_{\mtx{T}}(j,i)^2\right]\\
&= \sum_{n=1}^\infty\sum_{j=1}^{\ell}\sigma_{k+n}^2 \mtx{U}_{\mtx{T}}(j,i)^2\E\left[\left|\mtx{\Omega}_2(n,j)\right|^2\right],
\end{align*}
where $\sigma_{k+1},\sigma_{k+2},\ldots$ are the diagonal elements of $\mtx{\Sigma}_2$. Then, the quasimatrix $\mtx{\Omega}_2$ has independent columns and, using \cref{prop_cov_matrix}, we have 
\[\E\left[|\mtx{\Omega}_2(n,j)|^2\right] = \int_{D_1\times D_1}v_{k+n}(x)K(x,y)v_{k+n}(y)\d x\d y,\] where $v_{k+n}$ is the $n$th column of $\mtx{V}_2$. Then, $\E\left[|\mtx{\Omega}_2(n,j)|^2\right]\leq \lambda_1$, as $\E\left[|\mtx{\Omega}_2(n,j)|^2\right]$ is written as a Rayleigh quotient. Finally, we have 
\[\E\left[\|\mtx{\Sigma}_2 \mtx{V}_2^* \mtx{\Omega} \mtx{T}\|_{\HS}^2\right]\leq \lambda_1\sum_{i=1}^k ((\mtx{D}_{\mtx{T}})_{ii})^2\sum_{j=1}^{\ell}\mtx{U}_{\mtx{T}}(j,i)^2\sum_{n=1}^\infty\sigma_{k+n}^2 = \lambda_1 \|\mtx{T}\|_{\textup{F}}^2\|\mtx{\Sigma}_2\|_{\HS}^2,\] 
by orthonormality of the columns on $\mtx{U}_{\mtx{T}}$.
\end{proof}

We are now ready to prove~\cref{th_tropp_random_svd_Frob}, which shows that the randomized SVD can be generalized to HS operators. 
\begin{proof}[Proof of Theorem~\ref{th_tropp_random_svd_Frob}]
Let $\mtx{\Omega}_1, \mtx{\Omega}_2$ be the quasimatrices defined in~\cref{th_tropp_deter_svd}. The $k\times (k+p)$ matrix $\mtx{\Omega}_1$ has full rank with probability one and by \cref{th_tropp_deter_svd}, we have
\begin{align*}
\E\left[\|(\mtx{I}-\mtx{P}_\mtx{Y})\F\|_{\HS}\right] &\leq \E\left[\left(\|\mtx{\Sigma}_2\|_{\HS}^2+ \|\mtx{\Sigma}_2\mtx{\Omega}_2\mtx{\Omega}_1^{\dagger}\|_{\HS}^2\right)^{1/2}\right]\leq \|\mtx{\Sigma}_2\|_{\HS} + \E \|\mtx{\Sigma}_2\mtx{\Omega}_2\mtx{\Omega}_1^{\dagger}\|_{\HS}\\
&\leq \|\mtx{\Sigma}_2\|_{\HS} + \E \left[\|\mtx{\Sigma}_2\mtx{\Omega}_2\|_{\HS}^2\right]^{1/2}\E \left[\|\mtx{\Omega}_1^\dagger\|_{\textup{F}}^2\right]^{1/2},
\end{align*}
where the last inequality follows from Cauchy--Schwarz inequality. Then, using \cref{prop_rand_1} and \cref{prop_rand_2}, we have
\[\E \left[\|\mtx{\Sigma}_2\mtx{\Omega}_2\|_{\HS}^2\right]\leq \lambda_1(k+p)\|\mtx{\Sigma}_2\|_{\HS}^2,\qquad\text{and} \qquad \E\left[\|\mtx{\Omega}_1\|^2_{\textup{F}}\right]\leq \frac{1}{\gamma_k\lambda_1}\frac{k}{p-1}.\]
where $\gamma_k$ is defined in \cref{sec_quality_kernel}.
The observation that $\|\mtx{\Sigma}_2\|_{\HS}^2=\sum_{j=k+1}^\infty \sigma_j^2$ concludes the proof of \cref{eq:MainExpectationBound}.

For the probabilistic bound in \cref{eq:MainProbabilityBound}, we note that by \cref{th_tropp_deter_svd} we have,
\[
\|\F-\mtx{P}_\mtx{Y}\F\|_{\HS}^2\leq \|\mtx{\Sigma}_2\|_{\HS}^2+\|\mtx{\Sigma}_2\mtx{\Omega}_2\mtx{\Omega}_1^{\dagger}\|_{\HS}^2\leq (1+\|\mtx{\Omega}_2\|_{\HS}^2\|\mtx{\Omega}_1^{\dagger}\|_{\textup{F}}^2)\|\mtx{\Sigma}_2\|_{\HS}^2,
\]
where the second inequality uses the submultiplicativity of the HS norm. The bound follows from bounding $\|\mtx{\Omega}_1^\dagger\|_{\textup{F}}^2$ and $\|\mtx{\Omega}_2\|_{\HS}^2$ using \cref{thm_propa_bound,prop_control_matrices}, respectively.
\end{proof}

\section{Recovering the Green's function from input-output pairs} \label{sec_approx_Green}

It is known that the Green's function associated with~\cref{eq:PDEsimple} always exists, is unique, is a nonnegative function $G : D \times D \to \R^+ \cup \{\infty\}$ such that 
\[
u(x) =\int_D G(x,y)f(y)\d y, \qquad  f \in \mathcal{C}_c^\infty(D), 
\]
and for each $y\in\Omega$ and any $r>0$, we have $G(\cdot, y) \in \mathcal{H}^1(D \setminus B_r(y))\cap \mathcal{W}_0^{1,1}(D)$~\cite{gruter1982green}.\footnote{Here, $B_r(y) = \{z\in\mathbb{R}^3 : \|z - y\|_2 < r\}$, $\mathcal{W}^{1,1}(D)$ is the space of weakly differentiable functions in the $L^1$-sense, and $\mathcal{W}^{1,1}_0(D)$ is the closure of $\mathcal{C}_c^\infty(D)$ in $\mathcal{W}^{1,1}(D)$.}
Since the PDE in~\cref{eq:PDEsimple} is self-adjoint, we also know that for almost every $x,y\in D$, we have $G(x,y) = G(y,x)$~\cite{gruter1982green}. 

We now state \cref{th_Green}, which shows that if $N = \mathcal{O}(\epsilon^{-6}\log^4(1/\epsilon))$ and one has $N$ input-output pairs $\{(f_j,u_j)\}_{j=1}^N$ with algorithmically-selected $f_j$, then the Green's function associated with $\mathcal{L}$ in \cref{eq:PDEsimple} can be recovered to within an accuracy of $\mathcal{O}(\Gamma_\epsilon^{-1/2}\log^3(1/\epsilon)\epsilon)$ with high probability. Here, the quantity $0<\Gamma_\epsilon\leq 1$ measures the quality of the random input functions $\{f_j\}_{j=1}^N$ (see~\cref{sec_green_reconst_adm}).

\begin{theorem} \label{th_Green}
Let $0<\epsilon<1$, $D\subset \R^3$ be a bounded Lipschitz domain, and $\L$ given in~\cref{eq:PDEsimple}. If $G$ is the Green's function associated with $\L$, then there is a randomized algorithm that constructs an approximation $\tilde{G}$ of $G$ using $\mathcal{O}(\epsilon^{-6}\log^4(1/\epsilon))$ input-output pairs such that, as $\epsilon\rightarrow 0$, we have
\begin{equation} \label{eq_proba_Green_bound}
\|G-\tilde{G}\|_{L^2(D\times D)} = \mathcal{O} \left(\Gamma_\epsilon^{-1/2}\log^{3}(1/\epsilon)\epsilon\right)\|G\|_{L^2(D\times D)},
\end{equation}
with probability $\geq 1 - \mathcal{O}(\epsilon^{\log(1/\epsilon)-6})$. The term $\Gamma_\epsilon$ is defined by~\cref{eq_define_gamma_eps}.
\end{theorem}

Our algorithm that leads to the proof of \cref{th_Green} relies on the extension of the randomized SVD to HS operator (see \cref{sec_random_SVD}) and a hierarchical partition of the domain of $G$ into ``well-separated" domains. 

\subsection{Recovering the Green's function on admissible domains} \label{sec_exist_Green}
Roughly speaking, as $\|x-y\|_2$ increases $G$ becomes smoother about $(x,y)$, which can be made precise using so-called admissible domains~\cite{ballani2016matrices,bebendorf2008hierarchical,hackbusch2015hierarchical}. Let $\diam X\coloneq\sup_{x,y\in X}\|x-y\|_2$ be the diameter of $X$, $\dist(X,Y)\coloneq\inf_{x\in X,y\in Y}\|x-y\|_2$ be the shortest distance between $X$ and $Y$, and $\rho>0$ be a fixed constant. If $X,Y\subset \R^3$ are bounded domains, then we say that $X\times Y$ is an admissible domain if $\dist(X,Y)\geq \rho \max \{\diam X, \diam Y\}$; otherwise, we say that $X\times Y$ is non-admissible. There is a weaker definition of admissible domains as $\dist(X,Y)\geq \rho \min \{\diam X, \diam Y\}$~\cite[p.~59]{hackbusch2015hierarchical}, but we do not consider it.

\subsubsection{Approximation theory on admissible domains}\label{sec:ApproxTheoryAdmissibleDomain}
It turns out that the Green's function associated with~\cref{eq:PDEsimple} has rapidly decaying singular values when restricted to admissible domains. Roughly speaking, if $X,Y\subset D$ are such that $X\times Y$ is an admissible domain, then $G$ is well-approximated by a function of the form~\cite{bebendorf2003existence}
\begin{equation} \label{eq_approx_G}
G_k(x,y) = \sum_{j=1}^k g_j(x)h_j(y), \qquad (x,y)\in X\times Y,
\end{equation}
for some functions $g_1,\ldots,g_k\in L^2(X)$ and $h_1,\ldots,h_k\in L^2(Y)$. This is summarized in \cref{theo_bebendorf}, which is a corollary of~\cite[Thm.~2.8]{bebendorf2003existence}. 
\begin{theorem} \label{theo_bebendorf}
Let $G$ be the Green's function associated with~\cref{eq:PDEsimple} and $\rho>0$. Let $X,Y\subset D$ such that $\dist(X,Y)\geq \rho\max\{\diam X, \diam Y\}$.  Then, for any $0<\epsilon<1$, there exists $k\leq k_\epsilon\coloneq\lceil c(\rho,\diam D,\kappa_C)\rceil \lceil\log (1/\epsilon)\rceil^{4}+\lceil\log (1/\epsilon)\rceil$ and an approximant, $G_k$, of $G$ in the form given in~\cref{eq_approx_G} such that
\[
\|G-G_k\|_{L^2(X\times Y)}\leq \epsilon \|G\|_{L^2(X\times \hat{Y})}, \qquad \hat{Y}\coloneq\{y\in D,\, \dist(y,Y)\leq \frac{\rho}{2}\diam Y\},
\]
where $\kappa_C = \lambda_{\max}/\lambda_{\min}$ is the spectral condition number of the coefficient matrix $A(x)$ in~\cref{eq:PDEsimple}\footnote{Here, $\lambda_{\max}$ is defined as $\sup_{x\in D}\lambda_{\max}(A(x))$ and $\lambda_{\min} = \inf_{x\in D}\lambda_{\min}(A(x))>0$.} and $c$ is a constant that only depends on $\rho$, $\diam D$, $\kappa_C$. 
\end{theorem}
\begin{proof} 
In~\cite[Thm.~2.8]{bebendorf2003existence}, it is shown that if $Y=\tilde{Y}\cap D$ and $\tilde{Y}$ is convex, then there exists $k\leq c_{\rho/2}^3\lceil\log (1/\epsilon)\rceil^{4}+\lceil\log (1/\epsilon)\rceil$ and an approximant, $G_k$, of $G$ such that 
\begin{equation} \label{sec_sep_approx}
\|G(x,\cdot)-G_k(x,\cdot)\|_{L^2(Y)}\leq \epsilon \|G(x,\cdot)\|_{L^2(\hat{Y})}, \qquad x\in X,
\end{equation}
where $\hat{Y}\coloneq\{y\in D,\, \dist(y,Y)\leq \frac{\rho}{2}\diam Y\}$ and $c_{\rho/2}$ is a constant that only depends on $\rho$, $\diam Y$, and $\kappa_C$. As remarked by~\cite{bebendorf2003existence}, $\tilde{Y}$ can be included in a convex of diameter $\diam D$ that includes $D$ to obtain the constant $c(\rho,\diam D,\kappa_C)$. The statement follows by integrating the error bound in~\cref{sec_sep_approx} over $X$.
\end{proof} 

Since the truncated SVD of $G$ on $X\times Y$ gives the best rank $k_\epsilon\geq k$ approximation to $G$,~\cref{theo_bebendorf} also gives bounds on singular values:
\begin{equation} \label{eq_tail_SVD_error}
\left(\sum\nolimits_{j=k_\epsilon+1}^\infty\sigma_{j,X\times Y}^2\right)^{1/2} \leq \|G-G_{k}\|_{L^2(X\times Y)} \leq \epsilon \|G\|_{L^2(X\times \hat{Y})},
\end{equation}
where $\sigma_{j,X\times Y}$ is the $j$th singular value of $G$ restricted to $X\times Y$. Since $k_\epsilon = \mathcal{O}(\log^4(1/\epsilon))$, we conclude that the singular values of $G$ restricted to admissible domains $X\times Y$ rapidly decay to zero. 

\subsubsection{Randomized SVD for admissible domains} \label{sec_adm_domain}

Since $G$ has rapidly decaying singular values on admissible domains $X\times Y$, we use the randomized SVD for HS operators to learn $G$ on $X\times Y$ with high probability (see~\cref{sec_random_SVD}).

We start by defining a GP on the domain $Y$. Let $\mathcal{R}_{Y\times Y}K$ be the restriction\footnote{We denote the restriction operator by $\mathcal{R}_{Y\times Y}:L^2(D\times D)\to L^2(Y\times Y)$.} of the covariance kernel $K$ to the domain $Y\times Y$, which is a continuous symmetric positive definite kernel so that $\mathcal{GP}(0,\mathcal{R}_{Y\times Y}K)$ defines a GP on $Y$. We choose a target rank $k\geq 1$, an oversampling parameter $p\geq 2$, and form a quasimatrix $\mtx{\Omega} = \begin{bmatrix}f_1\,|\,\cdots \,|\, f_{k+p}\end{bmatrix}$ such that $f_j\in L^2(Y)$ and $f_j\sim \mathcal{GP}(0,\mathcal{R}_{Y\times Y} K)$ are identically distributed and independent. We then extend by zero each column of $\mtx{\Omega}$ from $L^2(Y)$ to $L^2(D)$ by $\mathcal{R}_Y^*\mtx{\Omega}=\begin{bmatrix}\mathcal{R}_Y^* f_1\,|\,\cdots \,|\, \mathcal{R}_Y^* f_{k+p}\end{bmatrix}$, where $\mathcal{R}_Y^* f_j\sim \mathcal{GP}(0,\mathcal{R}_{Y\times Y}^*\mathcal{R}_{Y\times Y}K)$. The zero extension operator $\mathcal{R}_Y^*:L^2(Y)\to L^2(D)$ is the adjoint of $\mathcal{R}_Y:L^2(D)\to L^2(Y)$.

Given the training data, $\mtx{Y} = \begin{bmatrix}u_1\,|\,\cdots \,|\, u_{k+p} \end{bmatrix}$ such that $\mathcal{L}u_j = \mathcal{R}_Y^* f_j$ and $u_j|_{\partial D} = 0$, we now construct an approximation to $G$ on $X\times Y$ using the randomized SVD (see~\cref{sec_random_SVD}). Following \cref{th_tropp_random_svd_Frob}, we have the following approximation error for $t\geq 1$ and $s\geq 2$:
\begin{equation} \label{eq_approx_error_loc_Green}
\|G-\tilde{G}_{X\times Y}\|_{L^2(X\times Y)}^2 \leq \left(1+t^2s^2\frac{3}{\gamma_{k,X\times Y}}\frac{k(k+p)}{p+1}\sum_{j=1}^\infty\frac{\lambda_j}{\lambda_1}\,\right)\left(\sum\nolimits_{j=k+1}^\infty\sigma_{j,X\times Y}^2\right)^{1/2},
\end{equation}
with probability greater than $1-t^{-p}-e^{-s^2(k+p)}$. Here, $\lambda_1\geq \lambda_2\geq\cdots>0$ are the eigenvalues of $K$, $\tilde{G}_{X\times Y} = \mtx{P}_{\mathcal{R}_{X}\mtx{Y}}\mathcal{R}_{X}\F\mathcal{R}_{Y}^*$ and $\mtx{P}_{\mathcal{R}_{X}\mtx{Y}} = \mathcal{R}_{X}\mtx{Y}((\mathcal{R}_{X}\mtx{Y})^*\mathcal{R}_{X}\mtx{Y})^{\dagger}(\mathcal{R}_{X}\mtx{Y})^*$ is the orthogonal projection onto the space spanned by the columns of $\mathcal{R}_{X}\mtx{Y}$. Moreover, $\gamma_{k,X\times Y}$ is a measure of the quality of the covariance kernel of $\mathcal{GP}(0,\mathcal{R}_{Y\times Y}^*\mathcal{R}_{Y\times Y}K)$ (see \cref{sec_quality_kernel}) and, for $1\leq i,j\leq k$, defined as $
\gamma_{k,X\times Y} = k/(\lambda_1\Tr(\mtx{C}_{X\times Y}^{-1}))$, where
\[[\mtx{C}_{X\times Y}]_{ij} = \int_{D\times D} \mathcal{R}_Y^*v_{i,X\times Y}(x)K(x,y) \mathcal{R}_Y^*v_{j,X\times Y}(y)\d x\d y,\]
and $v_{1,X\times Y},\ldots, v_{k,X\times Y}\in L^2(Y)$ are the first $k$ right singular functions of $G$ restricted to $X\times Y$. 

Unfortunately, there is a big problem with the formula $\tilde{G}_{X\times Y} = \mtx{P}_{\mathcal{R}_X\mtx{Y}}\mathcal{R}_{X}\F\mathcal{R}_{Y}^*$. It cannot be formed because we only have access to input-output data, so we have no mechanism for composing $\mtx{P}_{\mathcal{R}_X\mtx{Y}}$ on the left of $\mathcal{R}_{X}\F\mathcal{R}_{Y}^*$. Instead, we note that since the partial differential operator in~\cref{eq:PDEsimple} is self-adjoint, $\F$ is self-adjoint, and $G$ is itself symmetric. That means we can use this to write down a formula for $\tilde{G}_{Y\times X}$ instead. That is, 
\[
\tilde{G}_{Y\times X} = \tilde{G}_{X\times Y}^* = \mathcal{R}_{Y}\F\mathcal{R}_{X}^*\mtx{P}_{\mathcal{R}_X\mtx{Y}}, 
\]
where we used the fact that $\mtx{P}_{\mathcal{R}_X\mtx{Y}}$ is also self-adjoint. This means we can construct $\tilde{G}_{Y\times X}$ by asking for more input-output data to assess the quasimatrix $\F(\mathcal{R}_X^*\mathcal{R}_X\mtx{Y})$. Of course, to compute $\tilde{G}_{X\times Y}$, we can swap the roles of $X$ and $Y$ in the above argument. 

With a target rank of $k=k_\epsilon = \lceil c(\rho,\diam D,\kappa_C)\rceil\lceil\log (1/\epsilon)\rceil^{4}+\lceil\log (1/\epsilon)\rceil$ and an oversampling parameter of $p = k_\epsilon$, we can combine~\cref{theo_bebendorf} and \cref{eq_tail_SVD_error,eq_approx_error_loc_Green} to obtain the bound 
\[\|G-\tilde{G}_{X\times Y}\|_{L^2(X\times Y)}^2 \leq \left(1+t^2s^2\frac{6k_\epsilon}{\gamma_{k_{\epsilon},X\times Y}}\sum_{j=1}^\infty\frac{\lambda_j}{\lambda_1}\,\right)\epsilon^2\|G\|_{L^2(X\times \hat{Y})}^2 ,\]
with probability greater than $1-t^{-k_{\epsilon}}-e^{-2s^2 k_\epsilon}$. A similar approximation error holds for $\tilde{G}_{Y\times X}$ without additional evaluations of $\F$. We conclude that our algorithm requires $N_{\epsilon, X\times Y} \!= 2(k_\epsilon+p) =  \mathcal{O}\!\left(\log^4(1/\epsilon)\right)$ input-output pairs to learn an approximant to $G$ on $X\times Y$ and $Y\times X$.

\subsection{Ignoring the Green's function on non-admissible domains} \label{sec_lp_estimates}
When the Green's function is restricted to non-admissible domains, its singular values may not decay. Instead, to learn $G$ we take advantage of the off-diagonal decay property of $G$.
It is known that for almost every $x\neq y\in D$ then
\begin{equation} \label{eq_widman_green}
G(x,y)\leq \frac{c_{\kappa_C}}{\|x-y\|_2}\|G\|_{L^2(D\times D)},
\end{equation}
where $c_{\kappa_C}$ is an implicit constant that only depends on $\kappa_C$ (see~\cite[Thm.~1.1]{gruter1982green}).\footnote{Note that we have normalized~\cite[Eq.~1.8]{gruter1982green} to highlight the dependence on $\|G\|_{L^2(D\times D)}$.}

If $X\times Y$ is a non-admissible domain, then for any $(x,y)\in X\times Y$, we find that 
\[
\|x-y\|_2\leq \dist(X,Y)+\diam(X)+\diam(Y)< (2+\rho)\max\{\diam X,\diam Y\},
\]
because $\dist(X,Y)<\rho\max\{\diam X,\diam Y\}$. This means that $x\in B_r(y)\cap D$, where $r=(2+\rho)\max\{\diam X,\diam Y\}$. Using~\cref{eq_widman_green}, we have
\begin{align*}
\int_X G(x,y)^2 dx &\leq \int_{B_r(y)\cap D}G(x,y)^2\d x \leq c_{\kappa_C}^2\|G\|_{L^2(D\times D)}^2 \int_{B_r(y)}\|x-y\|_2^{-2}\d x\\
&\leq 4\pi c_{\kappa_C}^2 r\|G\|_{L^2(D\times D)}^2.
\end{align*}
Noting that $\diam(Y)\leq r/(2+\rho)$ and $\int_Y 1 \d y\leq 4\pi ({\rm diam}(Y)/2)^3/3$, we have the following inequality for non-admissible domains $X\times Y$:
\begin{equation} \label{eq_norm_green_non_adm}
\|G\|_{L^2(X\times Y)}^2\leq \frac{2\pi^2}{3(2+\rho)^3} c_{\kappa_C}^2 r^4 \|G\|_{L^2(D\times D)}^2,
\end{equation}
where $r=(2+\rho)\max\{\diam X,\diam Y\}$.
We conclude that the Green's function restricted to a non-admissible domain has a relatively small norm when the domain itself is small. Therefore, in our approximant $\tilde{G}$ for $G$, we ignore $G$ on non-admissible domains by setting $\tilde{G}$ to be zero.

\subsection{Hierarchical admissible partition of domain}\label{sec_hierar_domain}
We now describe a hierarchical partitioning of $D\times D$ so that many subdomains are admissible domains, and the non-admissible domains are all small. For ease of notion, we may assume---without loss of generality---that $\diam D = 1$ and $D\subset [0,1]^3$; otherwise, one should shift and scale $D$. Moreover, partitioning $[0,1]^3$ and restricting the partition to $D$ is easier than partitioning $D$ directly. For the definition of admissible domains, we find it convenient to select $\rho = 1/\sqrt{3}$. 

\begin{figure}
\centering
\vspace{0.5cm}
\begin{overpic}[width=\textwidth]{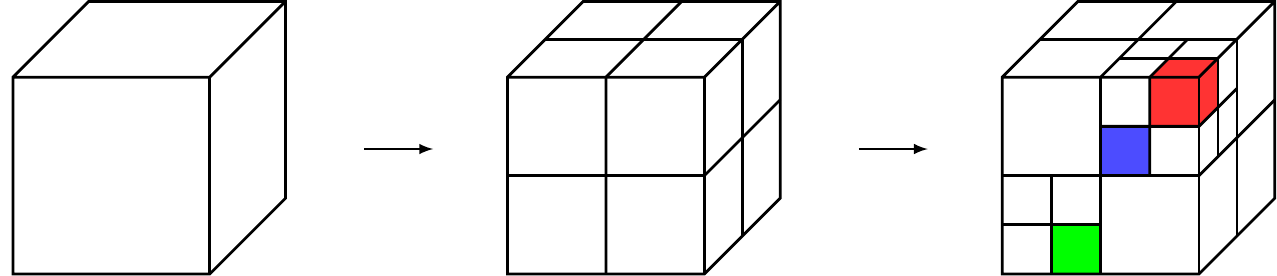}
\end{overpic}
\caption{Two levels of hierarchical partitioning of $[0,1]^3$. The blue and green domains are admissible, while the blue and red domains are non-admissible.}
\label{fig_hierar_tree}
\end{figure}

Let $I = [0,1]^3$. The hierarchical partitioning for $n$ levels is defined recursively as:
\begin{itemize}[leftmargin=*]
\item $I_{1\times 1\times 1}\coloneq I_1\times I_1\times I_1=[0,1]^3$ is the root for level $L = 0$.
\item At a given level $0\leq L\leq n-1$, if $I_{j_1\times j_2\times j_3}\coloneq I_{j_1}\times I_{j_2}\times I_{j_3}$ is a node of the tree, then it has $8$ children defined as
\[
\{I_{2 j_1+n_j(1)}\times I_{2 j_2+n_j(2)} \times I_{2j_3+n_j(3)}\mid n_j\in\{0,1\}^3\}.
\]
Here, if $I_j=[a,b]$, $0\leq a<b\leq 1$, then $I_{2j} = \left[a,\frac{a+b}{2}\right]$ and $I_{2j+1} = \left[\frac{a+b}{2},b\right]$.
\end{itemize}
The set of non-admissible domains can be given by this unwieldy expression 
\begin{equation} \label{eq_nonadm_d}
P_{\text{non-adm}} = \bigcup_{\substack{\bigwedge_{i=1}^3|j_i-\tilde{j}_i|\leq 1\\ 2^n\leq j_1,j_2,j_3\leq 2^{n+1}-1\\ 2^n\leq \tilde{j}_1,\tilde{j}_2,\tilde{j}_3\leq 2^{n+1}-1}} I_{j_1\times j_2\times j_3}\times I_{\tilde{j}_1\times \tilde{j}_2\times \tilde{j}_3},
\end{equation}
where $\land$ is the logical ``and'' operator. The set of admissible domains is given by 
\begin{equation} \label{eq_adm_d}
P_{\text{adm}} = \bigcup_{L=1}^n \Lambda(P_{\text{non-adm}}(L-1))\backslash P_{\text{non-adm}}(L)),
\end{equation}
where $P_{\text{non-adm}}(L)$ is the set of non-admissible domain for a hierarchical level of $L$ and
\[\Lambda(P_{\text{non-adm}}(L-1))=\bigcup_{\substack{I_{j_1\times j_2 \times j_3}\times I_{\tilde{j}_1\times\tilde{j}_2\times \tilde{j}_3}\\ \in P_{\text{non-adm}}(L-1)}}\, \bigcup_{n_j,n_{\tilde{j}}\in\{0,1\}^3}I_{\bigtimes_{i=1}^3 2j_i+n_j(i)}\times I_{\bigtimes_{i=1}^3 2 \tilde{j}_i+n_{\tilde{j}}(i)}.\]
Using \cref{eq_nonadm_d}-\cref{eq_adm_d}, the number of admissible and non-admissible domains are precisely $|P_{\text{non-adm}}| = (3\times 2^n-2)^3$ and $|P_{\text{adm}}| =  \sum_{\ell=1}^n 2^{6}(3\times 2^{L-1}-2)^3-(3\times 2^{L}-2)^3$.  In particular, the size of the partition at the hierarchical level $0\leq L\leq n$ is equal to $8^L$ and the tree has a total of $(8^{n+1}-1)/7$ nodes (see~\cref{fig_hierar_mat}). 

Finally, the hierarchical partition of $D\times D$ can be defined via the partition $P=P_{\text{adm}}\cup P_{\text{non-adm}}$ of $[0,1]^3$ by doing the following: 
\[
D\times D =\bigcup\limits_{\tau\times\sigma\in P} (\tau\cap D)\times(\sigma\cap D).
\]
The sets of admissible and non-admissible domains of $D\times D$ are denoted by $P_{\text{adm}}$ and $P_{\text{non-adm}}$ in the next sections.

\begin{figure}
\centering
\vspace{0.5cm}
\begin{overpic}[width=0.49\textwidth, trim=100 50 100 10, clip]{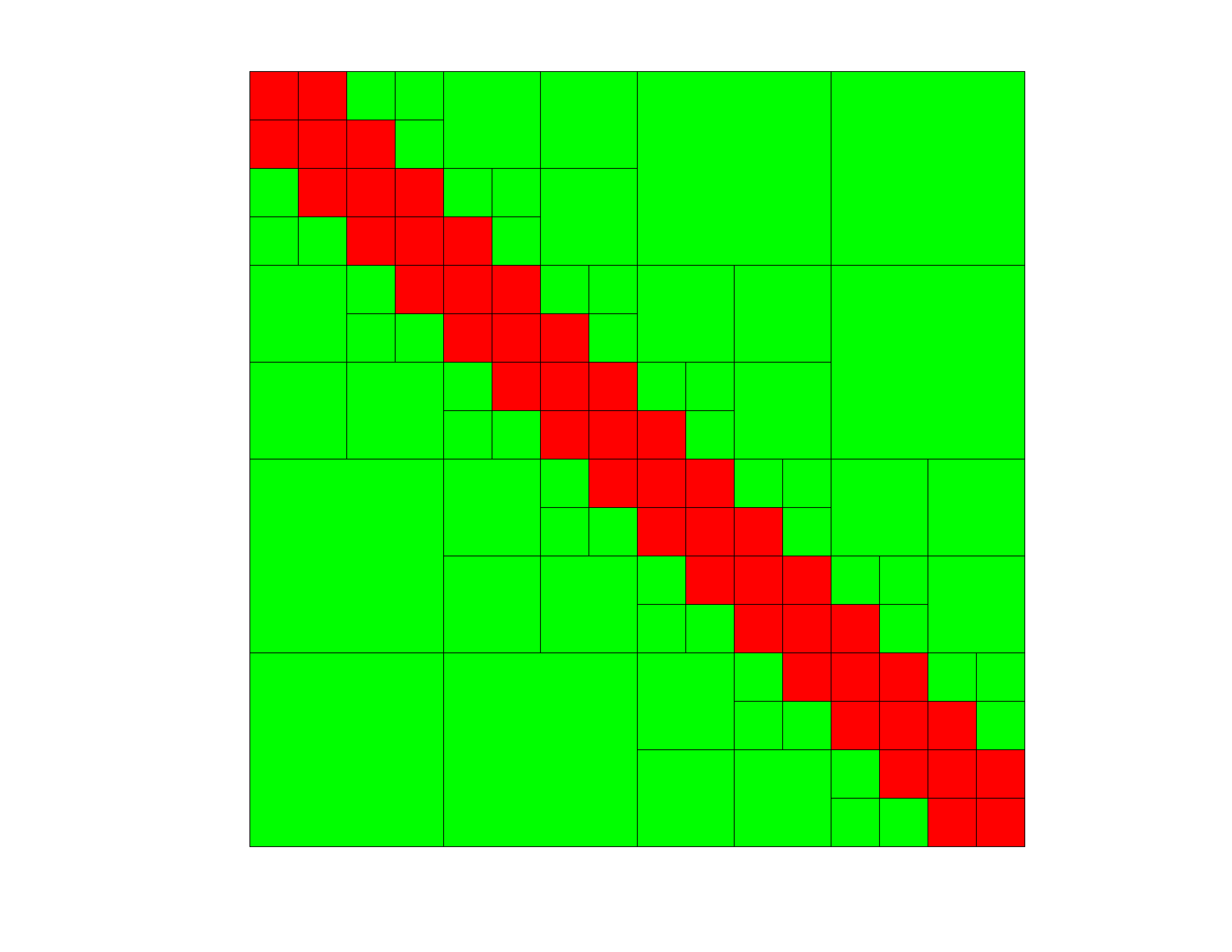}
\put(49,88){1D}
\end{overpic}
\begin{overpic}[width=0.49\textwidth, trim=100 50 100 10, clip]{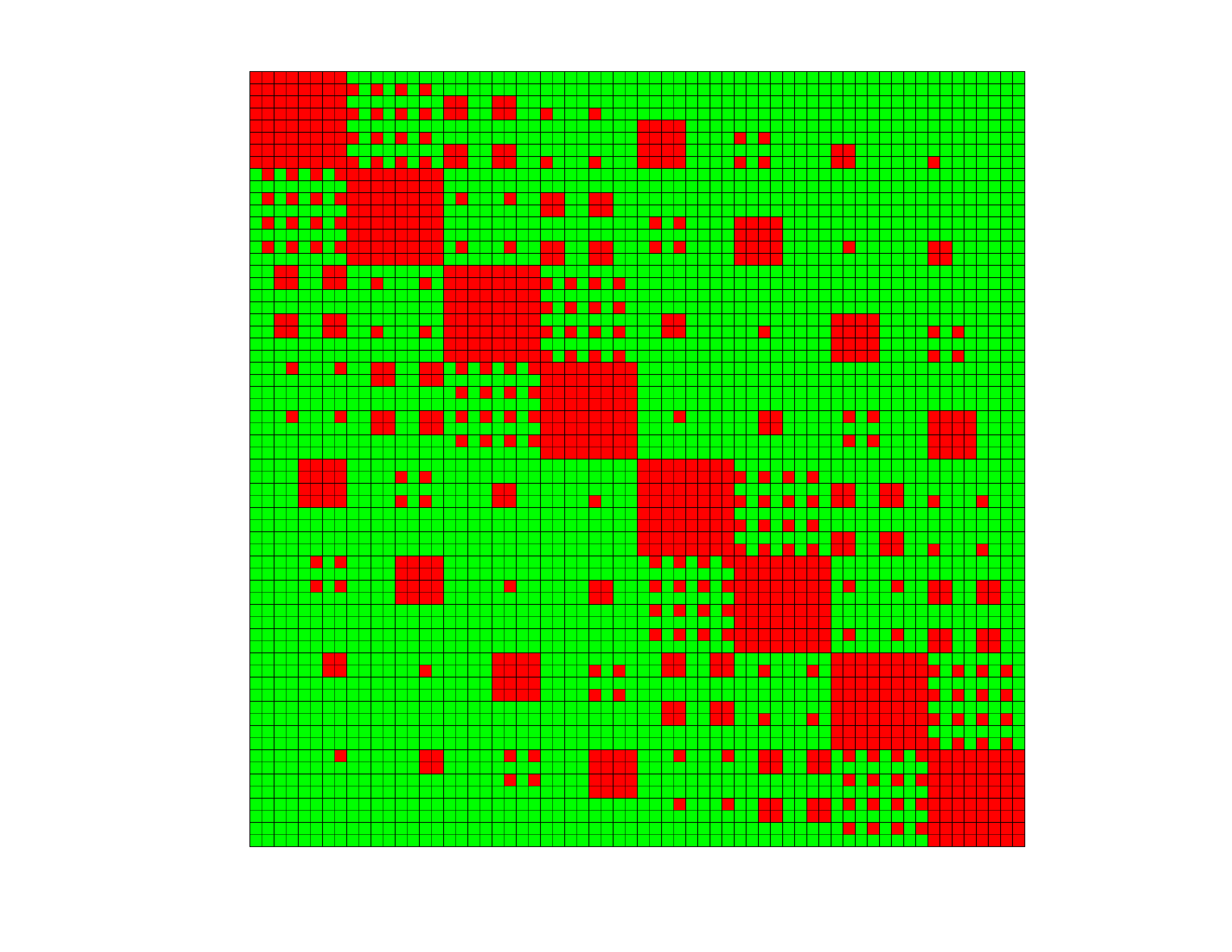}
\put(49,88){3D}
\end{overpic}
\caption{For illustration purposes, we include the hierarchical structure of the Green's functions in 1D after $4$ levels (left) and in 3D after $2$ levels (right). The hierarchical structure in 3D is complicated as this is physically a $6$-dimensional tensor that has been rearranged so it can be visualized.}
\label{fig_hierar_mat}
\end{figure}

\subsection{Recovering the Green's function on the entire domain} \label{sec_proof_approx_green}
We now show that we can recover $G$ on the entire domain $D\times D$. 

\subsubsection{Global approximation on the non-admissible set} \label{sec_approx_non_adm}
Let $n_\epsilon$ be the number of levels in the hierarchical partition $D\times D$ (see~\cref{sec_hierar_domain}). We want to make sure that the norm of the Green's function on all non-admissible domains is small so that we can safely ignore that part of $G$ (see~\cref{sec_lp_estimates}). As one increases the hierarchical partitioning levels, the volume of the non-admissible domains get smaller (see~\cref{fig_hierar_mat_1D}). 

\begin{figure}
\centering
\vspace{0.5cm}
\begin{overpic}[width=0.3\textwidth, trim=100 50 100 50, clip]{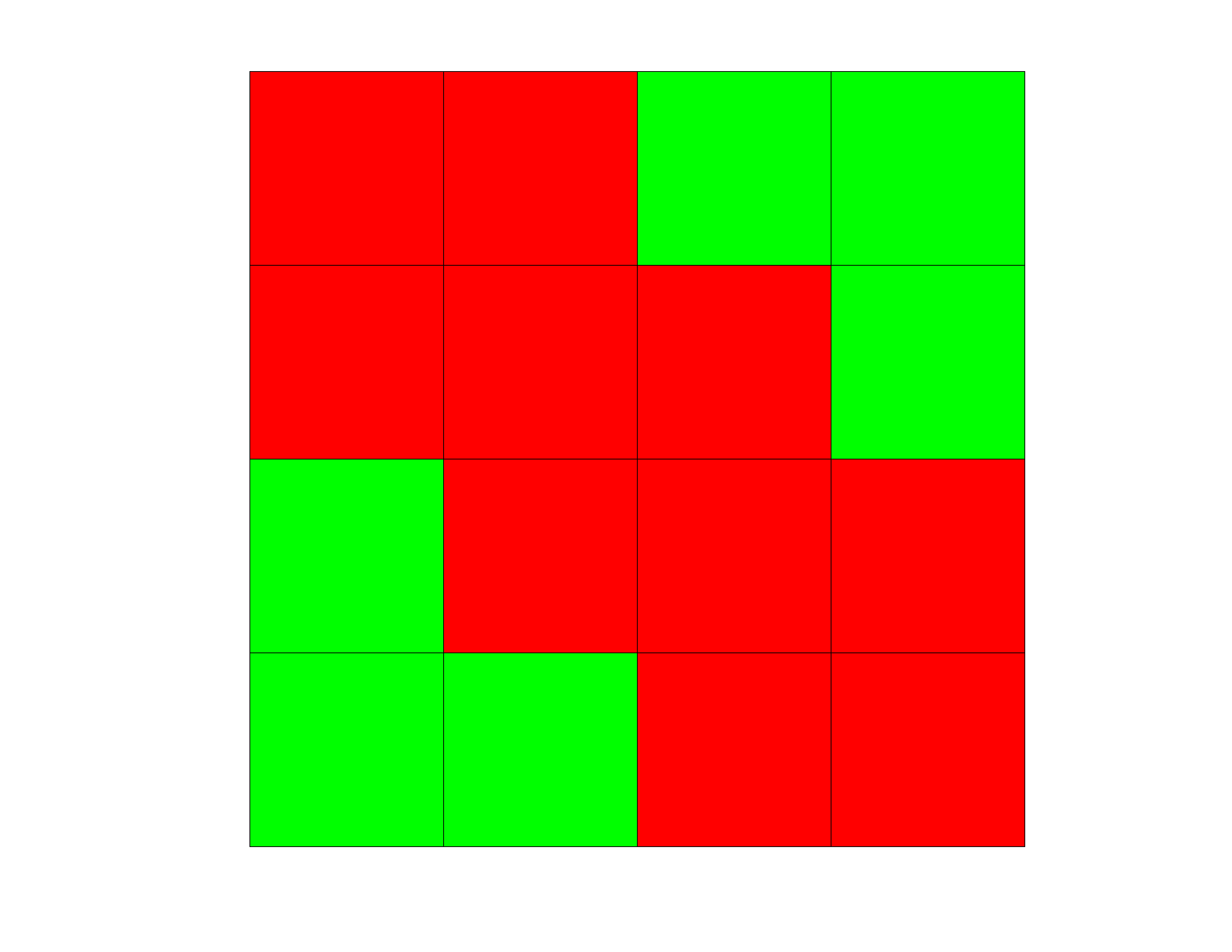}
\put(41,89){Level 2}
\end{overpic}
\begin{overpic}[width=0.3\textwidth, trim=100 50 100 50, clip]{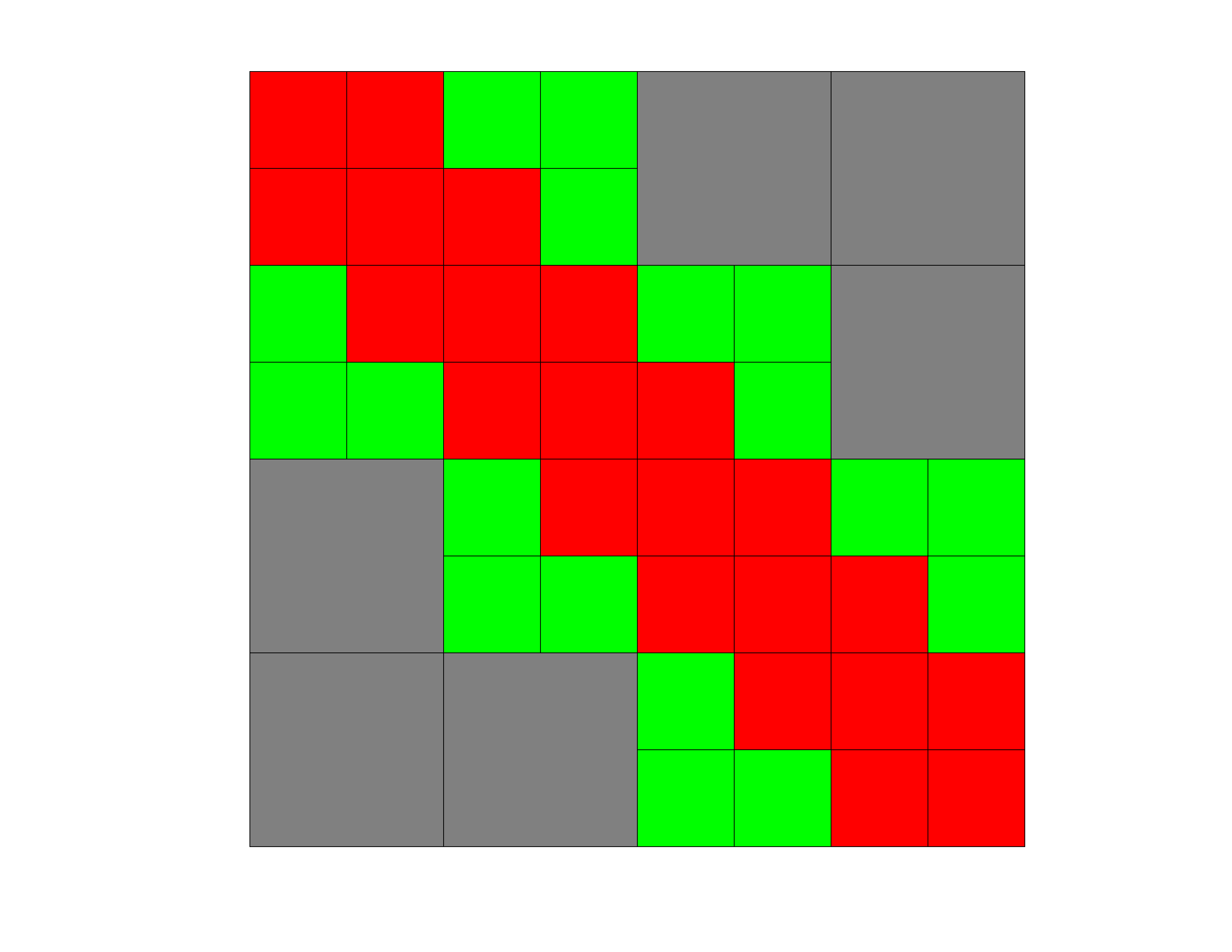}
\put(41,89){Level 3}
\end{overpic}
\begin{overpic}[width=0.3\textwidth, trim=100 50 100 50, clip]{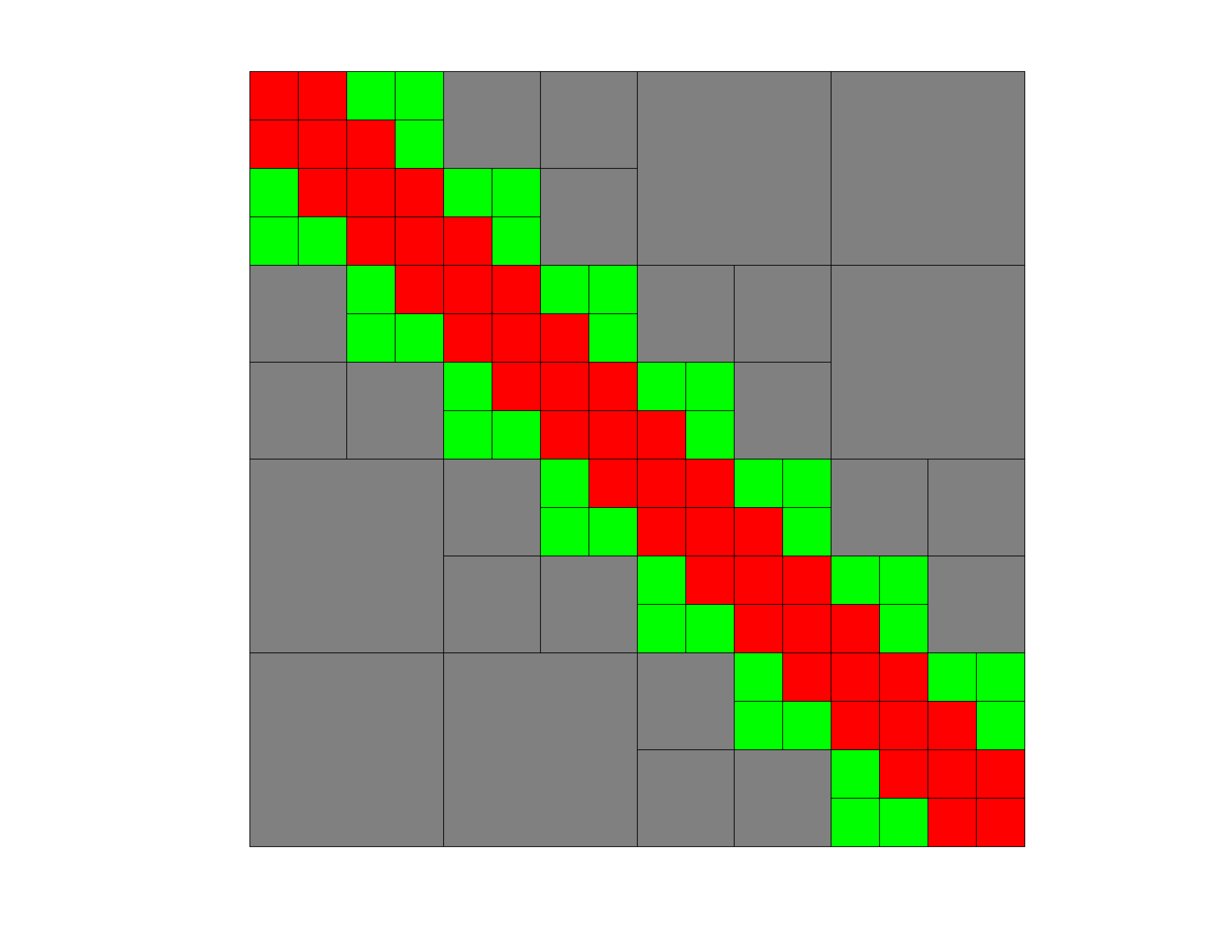}
\put(41,89){Level 4}
\end{overpic}
\caption{For illustration purposes, we include the hierarchical structure of the Green function in 1D. The green blocks are admissible domains at that level, the gray blocks are admissible at a higher level, and the red blocks are the non-admissible domains at that level. The area of the non-admissible domains decreases at deeper levels.}
\label{fig_hierar_mat_1D}
\end{figure}

Let $X\times Y\in P_{\text{non-adm}}$ be a non-admissible domain, the two domains $X$ and $Y$ have diameter bounded by $\sqrt{3}/2^{n_\epsilon}$ because they are included in cubes of side length $1/2^{n_\epsilon}$ (see \cref{sec_hierar_domain}). Combining this with~\cref{eq_norm_green_non_adm} yields
\[
\|G\|_{L^2(X\times Y)}^2\leq 2\pi^2(6+\sqrt{3})c_{\kappa_C}^2 2^{-4n_\epsilon}\|G\|_{L^2(D\times D)}^2.\]
Therefore, the $L^2$-norm of $G$ on the non-admissible domain $P_{\text{non-adm}}$ satisfies
\[\|G\|_{L^2(P_{\text{non-adm}})}^2 = \sum_{X\times Y\in P_{\text{non-adm}}}\|G\|_{L^2(X\times Y)}^2\leq 54\pi^2(6+\sqrt{3})c_{\kappa_C}^2 2^{-n_\epsilon}\|G\|_{L^2(D\times D)}^2,\]
where we used $|P_{\text{non-adm}}| = (3\times 2^{n_\epsilon}-2)^3\leq 27(2^{3n_\epsilon})$. 
This means that if we select $n_\epsilon$ to be
\begin{equation} \label{eq_level_epsilon}
n_\epsilon = \left\lceil \log_2(54\pi^2(6+\sqrt{3})c_{\kappa_C}^2)+2\log_2(1/\epsilon)\right\rceil \sim 2\log_2(1/\epsilon),
\end{equation}
then we guarantee that $\|G\|_{L^2(P_{\text{non-adm}})}\leq\epsilon\|G\|_{L^2(D\times D)}$.  We can safely ignore $G$ on non-admissible domains---by taking the zero approximant---while approximating $G$ to within $\epsilon$.

\subsubsection{Learning rate of the Green's function}
 \label{sec_green_reconst_adm}

Following \cref{sec_adm_domain}, we can construct an approximant $\tilde{G}_{X\times Y}$ to the Green's function on an admissible domain $X\times Y$ of the hierarchical partitioning using the HS randomized SVD algorithm, which requires $N_{\epsilon,X\times Y}=\smash{\mathcal{O}(\log^4(1/\epsilon))}$ input-output training pairs (see \cref{sec_adm_domain}). Therefore, the number of training input-output pairs needed to construct an approximant to $G$ on all admissible domains is given by 
\[
N_\epsilon = \sum_{X\times Y\in P_{\text{adm}}} N_{\epsilon,X\times Y} = \mathcal{O}\left(|P_{\text{adm}}|\log^4(1/\epsilon)\right),\]
where $|P_{\text{adm}}|$ denotes the total number of admissible domains at the hierarchical level $n_\epsilon$, which is given by \cref{eq_level_epsilon}. Then, we have (see~\cref{sec_hierar_domain}):
\begin{equation} \label{eq_number_admissible}
|P_{\text{adm}}| = \sum_{\ell=1}^{n_\epsilon} 2^{6}(3\times 2^{\ell-1}-2)^3-(3\times 2^{\ell}-2)^3 \leq 6^3 2^{3n_\epsilon},
\end{equation}
and, using~\cref{eq_level_epsilon}, we obtain $|P_{\text{adm}}| = \mathcal{O}(1/\epsilon^6)$. This means that the total number of required input-output training pairs to learn $G$ with high probability is bounded by 
\[
N_\epsilon = \mathcal{O}\left(\epsilon^{-6}\log^4(1/\epsilon)\right).
\]

\subsubsection{Global approximation error} \label{sec_proba_fail_Green}

We know that with $N_\epsilon = \mathcal{O}(\epsilon^{-6}\log^4(1/\epsilon))$ input-output training pairs, we can construct an accurate approximant to $G$ on each admissible and non-admissible domain. Since the number of admissible and non-admissible domains depends on $\epsilon$, we now check that this implies a globally accurate approximant that we denote by $\tilde{G}$. 

Since $\tilde{G}$ is zero on non-admissible domains and $P_{\text{adm}}\cap P_{\text{non-adm}}$ has measure zero, we have
\begin{equation} \label{eq_ineq_G}
 \|G-\tilde{G}\|_{L^2(D\times D)}^2
\leq \epsilon^2\|G\|_{L^2(D\times D)}^2+\sum_{X\times Y\in P_{\text{adm}}}\|G-\tilde{G}\|_{L^2(X\times Y)}^2.
\end{equation}
Following \cref{sec_green_reconst_adm}, if $X\times Y$ is admissible then the approximation error satisfies 
\[\|G-\tilde{G}_{X\times Y}\|_{L^2(X\times Y)}^2 \leq 12 t^2s^2\frac{k_\epsilon}{\gamma_{k_{\epsilon},X\times Y}}\sum_{j=1}^\infty\frac{\lambda_j}{\lambda_1}\epsilon^2\|G\|_{L^2(X\times \hat{Y})}^2 ,\]
with probability greater than $1-t^{-k_{\epsilon}}-e^{-2s^2 k_\epsilon}$. Here, $\hat{Y}=\{y\in D,\, \dist(y,Y)\leq \diam Y/2\sqrt{3}\}$ (see~\cref{theo_bebendorf} with $\rho=1/\sqrt{3}$). To measure the worst $\gamma_{k_\epsilon, X\times Y}$, we define
\begin{equation} \label{eq_define_gamma_eps}
\Gamma_\epsilon = \min\{\gamma_{k_\epsilon,X\times Y}:  X\times Y\in P_{\text{adm}}\}.
\end{equation}
From \cref{eq_lower_bound_gamma}, we know that $0<\Gamma_\epsilon\leq 1$ and that $1/\Gamma_\epsilon$ is greater than the harmonic mean of the first $k_\epsilon$ scaled eigenvalues of the covariance kernel $K$, i.e.,
\begin{equation} \label{eq_gamma_eps_bound}
\frac{1}{\Gamma_\epsilon} \geq \frac{1}{k_\epsilon}\sum_{j=1}^{k_\epsilon}\frac{\lambda_1}{\lambda_j},
\end{equation}

Now, one can see that $X\times \hat{Y}$ is included in at most $5^3=125$ neighbours including itself. 
Assuming that all the probability bounds hold on the admissible domains, this implies that
\begin{align*}
\sum_{X\times Y\in P_{\text{adm}}}\!\!\!\!\|G-\tilde{G}\|_{L^2(X\times Y)}^2 &\leq 
\sum_{X\times Y\in P_{\text{adm}}}\!\!\!\!\|G-\tilde{G}\|_{L^2(X\times Y)}^2 \leq 12 t^2s^2\frac{k_\epsilon}{\lambda_1\Gamma_{\epsilon}}\Tr(K)\epsilon^2\!\!\!\!\!\!\!\!\sum_{X\times Y\in P_{\text{adm}}}\!\!\!\!\|G\|_{L^2(X\times\hat{Y})}^2\\
&\leq 1500 t^2s^2\frac{k_\epsilon}{\lambda_1\Gamma_{\epsilon}}\Tr(K)\epsilon^2\|G\|^2_{L^2(D\times D)}.
\end{align*}
We then choose $t=e$ and $s=k_\epsilon^{1/4}$ so that the approximation bound on each admissible domain holds with probability of failure less than $2e^{-\sqrt{k_\epsilon}}$. Finally, using \cref{eq_ineq_G} we conclude that as $\epsilon \to 0$, the approximation error on $D\times D$ satisfies
\[\|G-\tilde{G}\|_{L^2(D\times D)} = \mathcal{O}\left(\Gamma_\epsilon^{-1/2}\log^3(1/\epsilon)\epsilon\right)\|G\|_{L^2(D\times D)},\]
with probability $\geq (1-2e^{-\sqrt{k_\epsilon}})^{6^3 2^{3n_\epsilon}}=1-\mathcal{O}(\epsilon^{\log(1/\epsilon)-6})$, where $n_\epsilon$ is given by \cref{eq_level_epsilon}. We conclude that the approximant $\tilde{G}$ is a good approximation to $G$ with very high probability.

\section{Conclusions and discussion} \label{sec_further_work}

This paper rigorously learns the Green's function associated with a PDE rather than the partial differential operator (PDO). By extending the randomized SVD to HS operators, we can identify a learning rate associated with elliptic PDOs in three dimensions and bound the number of input-output training pairs required to recover a Green's function approximately. One practical outcome of this work is a measure for the quality of covariance kernels, which may be used to design efficient kernels for PDE learning tasks.

There are several possible future extensions of these results related to the recovery of hierarchical matrices, the study of other partial differential operators, and practical deep learning applications, which we discuss further in this section.

\subsection{Fast and stable reconstruction of hierarchical matrices} \label{sec_stable_H_reconst}

We described an algorithm for reconstructing Green's function on admissible domains of a hierarchical partition of $D\times D$ that requires performing the HS randomized SVD $\mathcal{O}(\epsilon^{-6})$ times. We want to reduce it to a factor that is $\mathcal{O}(\text{polylog}(1/\epsilon))$.

For $n\times n$ hierarchical matrices, there are several existing algorithms for recovering the matrix based on matrix-vector products~\cite{boukaram2019randomized,lin2011fast,martinsson2011fast,martinsson2016compressing}. There are two main approaches: (1) The ``bottom-up'' approach: one begins at the lowest level of the hierarchy and moves up and (2) The ``top-down'' approach: one updates the approximant by peeling off the off-diagonal blocks and going down the hierarchy. The bottom-up approach requires $\mathcal{O}(n)$ applications of the randomized SVD algorithm~\cite{martinsson2011fast}. There are lower complexity alternatives that only require $\mathcal{O}(\log(n))$ matrix-vector products with random vectors~\cite{lin2011fast}. However, the algorithm in~\cite{lin2011fast} is not yet proven to be theoretically stable as errors from low-rank approximations potentially accumulate exponentially, though this is not observed in practice. For symmetric positive semi-definite matrices, it may be possible to employ a sparse Cholesky factorization~\cite{schafer2021sparse,schafer2017compression}. This leads us to formulate the following challenge: 

\fbox{\begin{minipage}[t][1.2\height][c]{\dimexpr\textwidth-15\fboxsep-2\fboxrule\relax}
\centering
{\bf Algorithmic challenge:} Design a provably stable algorithm that can recover an $n\times n$ hierarchical matrix using $\mathcal{O}(\log(n))$ matrix-vector products with high probability? 
\end{minipage}}

\medskip

If one can design such an algorithm and it can be extended to HS operators, then the $\mathcal{O}(\epsilon^{-6}\log^4(1/\epsilon))$ term in~\cref{th_Green} may improve to~$\mathcal{O}(\text{polylog}(1/\epsilon))$. This means that the learning rate of partial differential operators of the form of  \cref{eq:PDEsimple} will be a polynomial in $\log(1/\epsilon)$ and grow sublinearly with respect to $1/\epsilon$.

\subsection{Extension to other partial differential operators}
Our learning rate for elliptic PDOs in three variables (see~\cref{sec_approx_Green}) depends on the decay of the singular values of the Green's function on admissible domains~\cite{bebendorf2003existence}. We expect that one can also find the learning rate for other PDOs. 

It is known that the Green's functions associated to elliptic PDOs in two dimensions exist and satisfy the following pointwise estimate~\cite{dong2009green}:
\begin{equation} \label{eq_estimate_2d}
|G(x,y)|\leq C\left(\frac{1}{\gamma R^2}+\log\left(\frac{R}{\|x-y\|_2}\right)\right), \quad \|x-y\|_2\leq R\coloneq\frac{1}{2}\max(d_x,d_y),
\end{equation}
where $d_x=\dist(x,\partial D)$, $\gamma$ is a constant depending on the size of the domain $D$, and $C$ is an implicit constant. One can conclude that $G(x,\cdot)$ is locally integrable for all $x\in D$ with $\|G(x,\cdot)\|_{L^p(B_r(x)\cap D)}<\infty$ for $r>0$ and $1\leq p<\infty$.  We believe that the pointwise estimate in~\cref{eq_estimate_2d} implies the off-diagonal low-rank structure of $G$ here, as suggested in~\cite{bebendorf2003existence}. Therefore, we expect that the results in this paper can be extended to elliptic PDOs in two variables. 

PDOs in four or more variables are far more challenging since we rely on the following bound on the Green's function on non-admissible domains~\cite{gruter1982green}:
\[
G(x,y)\leq \frac{c(d,\kappa_C)}{\lambda_{\min}}\|x-y\|_2^{2-d}, \qquad x\neq y\in D,
\]
where $D\subset\R^d$, $d\geq 3$ is the dimension, and $c$ is a constant depending only on $d$ and $\kappa_C$. This inequality implies that the $L^p$-norm of $G$ on non-admissible domains is finite when $0\leq p < d/(d-2)$.  However, for a dimension $d\geq 4$, we have $p<2$ and one cannot ensure that the $L^2$ norm of $G$ is finite. Therefore, the Green's function may not be compatible with the HS randomized SVD.

It should also be possible to characterize the learning rate for elliptic PDOs with lower order terms (under reasonable conditions)~\cite{dong2020green,hwang2020green,kim2019green} and many parabolic operators~\cite{kim2020green} as the associated Green's functions have similar regularity and pointwise estimates. The main task is to extend~\cite[Thm.~2.8]{bebendorf2003existence} to construct separable approximations of the Green's functions on admissible domains. In contrast, we believe that deriving a theoretical learning rate for hyperbolic PDOs remains a significant research challenge for many reasons. The first roadblock is that the Green's function associated with hyperbolic PDOs do not necessarily lie in $L^2(D\times D)$. For example, the Green's function associated with the wave equation in three variables, i.e., $\L=\partial_t^2-\nabla^2$, is not square-integrable as
\[
G(x,t,y,s) = \frac{\delta(t-s-\|x-y\|_2)}{4\pi \|x-y\|_2},\qquad (x,t),(y,s)\in \R^3\times [0,\infty),
\]
where $\delta(\cdot)$ is the Dirac delta function.

\subsection{Connection with neural networks}
There are many possible connections between this work and neural networks (NNs) from practical and theoretical viewpoints. The proof of~\cref{th_Green} relies on the construction of a hierarchical partition of the domain $D\times D$ and the HS randomized SVD algorithm applied on each admissible domain. This gives an algorithm for approximating Green's functions with high probability. However, there are more practical approaches that currently do not have theoretical guarantees~\cite{feliu2020meta,gin2020deepgreen}. 

A promising opportunity is to design a NN that can learn and approximate Green's functions using input-output training pairs $\{(f_j,u_j)\}_{j=1}^N$~\cite{boulle2021data}. Once a neural network $\mathcal{N}$ has been trained such that $\|\mathcal{N}-G\|_{L^2}\leq \epsilon \|G\|_{L^2}$, the solution to $\L u = f$ can be obtained by computing the following integral:
\[
u(x) = \int_D \mathcal{N}(x,y)f(y)\d y.
\]
Therefore, this may give an efficient computational approach for discovering operators since a NN is only trained once. Incorporating a priori knowledge of the Green's function into the network architecture design could be particularly beneficial. One could also wrap the selection of the kernel in the GP for generating random functions and training data into a Bayesian framework. 

Finally, we wonder how many parameters in a NN are needed to approximate a Green's function associated with elliptic PDOs within a tolerance of $0<\epsilon<1$. Can one exploit the off-diagonal low-rank structure of Green's functions to reduce the number of parameters? We expect the recent work on the characterization of ReLU NNs' approximation power is useful~\cite{guhring2019error,petersen2018optimal,yarotsky2017error}. The use of NNs with high approximation power such as rational NNs might also be of interest to approximate the singularities of the Green's function near the diagonal~\cite{boulle2020rational}.

\begin{acknowledgements}
We want to thank Max Jenquin and Tianyi Shi for discussions. We also thank Matthew Colbrook, Abinand Gopal, Daniel Kressner, and Yuji Nakatsukasa for their feedback and suggestions on the paper. We are indebted to Christopher Earls for telling us about the idea of using Green's functions and Gaussian processes for PDE learning. We are grateful to Joel Tropp, whose suggestions led to sharper bounds for the randomized SVD, and the anonymous referees for their comments which improved the quality of the paper.
\end{acknowledgements}

\bibliographystyle{spmpsci}
\bibliography{references}

\end{document}